\documentclass[a4paper,12pt]{article}
\usepackage{caption}
\usepackage{amsmath,amsthm,amssymb}
\usepackage{times}
\usepackage{enumerate}
\usepackage{amssymb}
\usepackage[square,sort,comma,numbers]{natbib}
\usepackage{latexsym}
\usepackage{color,graphicx}
\usepackage{float}
\usepackage{url}

\RequirePackage[numbers]{natbib}
\RequirePackage[colorlinks,citecolor=blue,urlcolor=blue]{hyperref}

\def\titlerunning#1{\gdef\titrun{#1}}
\makeatletter
\def\author#1{\gdef\autrun{\def\and{\unskip, }#1}\gdef\@author{#1}}

\makeatother

\def\subjclass#1{{\renewcommand{\thefootnote}{}%
\footnote{\emph{Mathematics Subject Classification (2010):} #1}}}
\def\keywords#1{\par\medskip
\noindent\textbf{Keywords.} #1}

\newtheorem{theorem}{Theorem}[section]

\newtheorem{lemma}[theorem]{Lemma}
\newtheorem{proposition}[theorem]{Proposition}

\theoremstyle{definition}
\newtheorem{definition}[theorem]{Definition}
\newtheorem{remark}[theorem]{Remark}

\DeclareMathOperator{\supp}{supp}
\DeclareMathOperator{\law}{Law}
\DeclareMathOperator{\sgn}{sgn}
\DeclareMathOperator{\spann}{span}
\DeclareMathOperator{\pr}{pr}

\numberwithin{equation}{section}

\newcommand{\eps}{\varepsilon}
\newcommand{\R}{\mathbb{R}}
\newcommand{\I}{\mathbb{I}}
\newcommand{\Q}{\mathbb{Q}}

\newcommand{\N}{\mathbb{N}}
\newcommand{\p}{\mathbb{P}}
\newcommand{\Cf}{\mathrm{C}}
\newcommand{\B}{\mathcal{B}}

\newcommand{\leb}{\mathrm{Leb}}
\newcommand{\F}{\mathcal{F}}
\renewcommand{\P}[1]{\ensuremath{\mathbb P \left[ #1 \right]}}
\newcommand{\E}[1]{\ensuremath{\mathbb E \left[ #1 \right]}}
\newcommand{\LL}[1]{\ensuremath{#1^T}}
\newcommand{\M}{\mathcal{M}}
\newcommand{\Ito}{It\^o }
\newcommand{\e}{\ensuremath{\mathbb E}}
\newcommand{\nd}{\alpha_0}
\newcommand{\W}{\mathcal{W}}
\newcommand{\cN}{\mathcal{N}}
\newcommand{\cS}{\mathcal{S}}
\newcommand{\by}{\mathrm{y}}
\newcommand{\bX}{\mathrm{X}}
\renewcommand{\Im}{\mathrm{Im}\,}
\newcommand{\HS}{\mathcal{L}_2}
\renewcommand{\L}{\mathrm{L}}
\newcommand{\tpr}{\widetilde{\pr}}
\newcommand{\mo}[1]{\ensuremath{\left[ #1 \cdot\right] }}
\newcommand{\te}{\tilde{e}}
\newcommand{\rB}{\mathrm{B}}

\title{Sticky-Reflected Stochastic Heat Equation\\ Driven by Colored Noise
}
\titlerunning{Sticky-Reflected Stochastic Heat Equation}
\author{Vitalii Konarovskyi\footnote{Fakult\"{a}t f\"{u}r Mathematik und Informatik, Universit\"{a}t Leipzig, Augustusplatz 10, 04109 Leipzig, Germany, E-mail: \href{mailto:konarovskyi@gmail.com}{konarovskyi@gmail.com}}\, \footnote{Institute of Mathematics of NAS of Ukraine, Tereschenkivska st. 3, 01024 Kiev, Ukraine}} 

\date{\today}

\begin{document}

\maketitle

\abstract{ 
  We prove the existence of a sticky-reflected solution to the heat equation on the spatial interval $[0,1]$ driven by colored noise. The process can be interpreted as an infinite-dimensional analog of the sticky-reflected Brownian motion on the real line, but now the solution obeys the usual stochastic heat equation except points where it reaches zero. At zero the solution has no noise and a drift pushes it to stay positive. The proof is based on a new approach that can also be applied to other types of SPDEs with discontinuous coefficients. 
}



\section{Introduction}

In this paper, we study the existence of a continuous function $X:[0,1] \times [0,\infty)\to[0,\infty)$ that is a weak solution to the following SPDE
\begin{equation} 
  \label{equ_main_equation}
  \frac{\partial X_t}{\partial t}= \frac{1}{ 2 }\frac{\partial^2 X_t}{\partial u^2}+\lambda \I_{\left\{ X_t=0 \right\}}+f(X_t)+\I_{\left\{ X_t>0 \right\}}Q\dot{W}_t
\end{equation}
with Neumann 
\begin{equation} 
  \label{equ_neumann_boundary_conditions}
  X'_t(0)= X'_t(1)= 0,\quad t\geq 0,
\end{equation}
or Dirichlet
\begin{equation} 
  \label{equ_dirichlet_boundary_conditions}
  X_t(0)= X_t(1)=0, \quad t\geq 0,
\end{equation}
boundary conditions and the initial condition
\begin{equation} 
  \label{equ_initial_condition}
  X_0(u)= g(u),\quad u \in [0,1],
\end{equation}
where $\dot{W}$ is a space-time white noise, the functions $g\in \Cf[0,1]$ and $\lambda \in \L_2:=\L_2[0,1]$ are non-negative, $f$ is a continuous function from $[0,\infty)$ to $[0,\infty) $ which has a linear growth and $f(0)=0$, and $Q$ is a non-negative definite self-adjoint Hilbert-Schmidt operator on $\L_2$. We will also assume that in the case of the Dirichlet boundary conditions $g(0)=g(1)=0$. 

The equation appears as a sticky-reflected counterpart of the reflected SPDE introduced in~\cite{Haussmann:1989,Nualart:1992}. We assume that a solution obeys the stochastic heat equation being strictly positive, but reaching zero, its diffusion vanishes and an additional drift at zero pushes the process to be positive. The form of equation~\eqref{equ_main_equation} is similar to the form of the SDE for a sticky-reflected Brownian motion on the real line
\begin{equation} 
  \label{SRBM}
  dx(t)= \lambda \I_{\left\{ x(t)=0 \right\}}dt+\I_{\left\{ x(t)>0 \right\}}\sigma dw(t),
\end{equation}
and we expect that the local behaviour of $X$ at zero is very similar to the behaviour of the sticky-reflected Brownian motion $x$. Remark that SDE~\eqref{SRBM} admits only a weak unique solution because of its discontinuous coefficients, see e.g. \cite{Chitashvili:1997,Karatzas:2011,Engelbert:2014}. The approaches which are applicable to sticky processes in finite-dimensional spaces can not be used for solving of SPDE~\eqref{equ_main_equation}. For instance,  Engelbert and Peskir in~\cite{Engelbert:2014} showed that equation~\eqref{SRBM} admits a weak (unique) solution, where their approach was based on the time change for a reflected Brownian motion. This method is very restrictive and can be applied only for the sticky dynamics in a one-dimensional state space. An equation for sticky-reflected dynamics for higher (finite) dimensions was considered by Grothaus and co-authors in~\cite{Grothaus:2017,Fattler:2016,Grothaus:2018}, where they used the Dirichlet form approach~\cite{Fukushima:2011,Ma:1992}. This approach was based on a priori knowledge of the invariant measure. Since the space is infinite-dimensional in our case, finding of the invariant measure seems a very complicated problem (see e.g.~\cite{Funaki:2001,Zambotti:2001} for the form of invariant measure for the reflected stochastic heat equation driven by the white noise). 

In this paper, we propose a new method for the proof of existence of weak solutions to equations describing sticky-reflected behaviour. This approach is is a modification of the method proposed by the author in~\cite{Konarovskyi:CFWDPA:2017}, and is based on a property of quadratic variation of semimartingales.

The paper leaves a couple of important open problems. The first problem is the uniqueness of a solution to SPDE~\eqref{equ_main_equation}-\eqref{equ_initial_condition}. Similarly to the one-dimensional SDE for sticky-reflected Brownian motion~\eqref{SRBM}, where the strong uniqueness is failed~\cite{Chitashvili:1997,Engelbert:2014}, we do not expect the strong uniqueness for the SPDE considered here. However, we believe that the weak uniqueness holds.

Another interesting question is the existence of solutions to a similar sticky-reflected heat equation driven by the white-noise. It seems that the method proposed here can be adapted to the case of such an SPDE. For this we need a similar statement to Theorem~\ref{the_properties_of_quadratic_variation_of_l2_valued_processes}, that remains an open problem.

\subsection{Definition of solution and main result}%
\label{sub:definition_of_solution}

For convenience of notation we introduce a parameter $\nd$ which equals to 1 in the case of the Neumann boundary conditions~\eqref{equ_neumann_boundary_conditions} and 0 in the case of Dirichlet boundary conditions~\eqref{equ_dirichlet_boundary_conditions}. Let us also introduce for $k\geq 1$ the space $\Cf^k[0,1]$ of $k$-times continuously differentiable functions on $(0,1)$ which together with their derivatives up to the order $k$ can be extended to continuous functions on $[0,1]$. We will write $\varphi \in \Cf^k_{\alpha_0}[0,1]$ if additionally $\varphi^{(\alpha_0)}(0)=\varphi^{(\alpha_0)}(1)=0$, where $\varphi^{(0)}=\varphi$ and $\varphi^{(1)}=\varphi'$.

Denote the inner product in the space $\L_2$ by $\langle \cdot  , \cdot  \rangle$ and the corresponding norm by $\|\cdot\|$. Let us give a definition of a weak solution to SPDE~\eqref{equ_main_equation}. 

\begin{definition} 
  \label{def_definition_of_solutions_to_spde}
  We say that a continuous function $X:[0,1] \times [0,\infty) \to [0,\infty)$ is a {\it (weak) solution} to SPDE~\eqref{equ_main_equation}-\eqref{equ_initial_condition}, if for every $\varphi \in \Cf^{2}_{\alpha_0}[0,1]$ the process 
\begin{align*}
  \M^{\varphi}&=  \langle X_t , \varphi \rangle- \langle X_0 , \varphi \rangle - \frac{1}{ 2 }\int_{ 0 }^{ t } \langle X_s , \varphi'' \rangle ds\\
  &- \int_{ 0 }^{ t } \langle \lambda \I_{\left\{ X_s=0 \right\}} , \varphi \rangle ds-\int_{ 0 }^{ t } \langle f(X_s) , \varphi \rangle ds , \quad t \geq 0, 
\end{align*}
is an $(\F^{X}_t)$- martingale with quadratic variation 
\[
  \left[ M^{\varphi} \right]_t=\int_{ 0 }^{ t } \left\|Q (\I_{\left\{ X_s>0 \right\}}\varphi)\right\|^2ds, \quad t \geq 0.
\]
\end{definition}

Hereinafter $\{ e_k,\ k\geq 1 \}$ will denote the basis in $\L_2$ consisting of eigenvectors of the non-negative definite self-adjoint operator $Q$. Let $\{ \mu_k,\ k\geq 1 \}$ be the corresponding family of eigenvalues of $Q$. We note that $\sum_{ k=1 }^{ \infty } \mu_k^2<\infty$, since $Q$ is a Hilbert-Schmidt operator. Introduce the function 
\begin{equation} 
  \label{equ_function_chi}
  \chi^2:=\sum_{ k=1 }^{ \infty }\mu_k^2e_k^2,
\end{equation}
where the series trivially converges in $L^1[0,1]$ and a.e.  The main result of this paper reads as follows.

\begin{theorem}[Existence of solutions] 
  \label{the_existence_of_solution_to_spde}
  If 
  \begin{equation} 
  \label{equ_the_main_condition_for_existance}
    \lambda \I_{\left\{ \chi>0 \right\}}=\lambda \quad \mbox{a.e.},
  \end{equation}
  then SPDE~\eqref{equ_main_equation}-\eqref{equ_initial_condition} admits a weak solution.
\end{theorem}


\begin{remark} %
  Condition~\eqref{equ_the_main_condition_for_existance} means that the drift $\lambda$ has to be equal to zero for those $u$ for which the noise vanishes.
\end{remark}

\begin{remark} 
  \label{rem_existence_when_the_main_condition_faild}
  The equation can admit a solution even if condition~\eqref{equ_the_main_condition_for_existance} does not hold. The reason is that the existence can be failed if $X_t(u)=0$ for $u \in [0,1]$ such that $\lambda(u)>0$ and $\chi(u)=0$ because of the term $\lambda\I_{\{X_t=0\}}$ and the absence of the noise for such $u$. However, if the initial condition is strictly positive for such $u$, then the solution could stay always strictly positive for such $u$, by the comparison principle for the classical heat equation. Therefore, the solution will exist. Take for example $Q=0$ and $f=0$. Then a weak solution to the heat equation
  \[
    \frac{\partial X_t}{\partial t}= \frac{1}{ 2 }\frac{\partial^2 X_t}{\partial u^2}
  \]
  considered with corresponding boundary and initial conditions is a solution to SPDE \eqref{equ_main_equation}-\eqref{equ_initial_condition} if $X_t(u)>0$, $t>0$, $u \in (0,1)$. But the strong positivity of $X$ is valid e.g. under the assumption the strong positivity of the initial condition. Hence, SPDE~\eqref{equ_main_equation}-\eqref{equ_initial_condition} has a weak solution even if $\lambda>0$ for e.g. $Q=0$, $f=0$ and $g>0$.
\end{remark}

We will construct a solution to equation~\eqref{equ_main_equation} as a limit of polygonal approximation similarly to the approach done in~\cite{Funaki:1983}. The main difficulty here is that coefficients are discontinuous. So, we cannot pass to the limit directly. In the next section, we will explain the key idea which allows to overcome this difficulty.

\subsection{Key idea of passing to the limit}%
\label{sub:the_main_idea_of_passing_to_the_limit}

We demonstrate our idea of passing to the limits in the case of discontinuous coefficients using the equation for a sticky-reflected Brownian motion in $\R $
\begin{equation} 
  \label{equ_sticky_reflected_brownian_motion}
  \begin{split}
    dx(t)&= \lambda \I_{\left\{ x(t)=0 \right\}}dt+\I_{\left\{ x(t)>0 \right\}}\sigma dw(t), \quad t\geq 0,\\
    x(0)&= x^0,
  \end{split}
\end{equation}
where $w$ is a standard Brownian motion in $\R $ and $\lambda$, $\sigma$, $x^0$ are positive constants. It is well known that this equation has only a unique weak solution (see e.g.~\cite{Engelbert:2014}).

Let us show that a solution to SDE~\eqref{equ_sticky_reflected_brownian_motion} can be constructed as a weak limit of solutions to equations with ``good'' coefficients. The first three steps proposed here are rather standard and the last step shows how one can overcome the problem of the discontinuity of the coefficients.

{\it Step I. Approximating sequence.} Consider a non-decreasing continuously differentiable function $\kappa:\R \to \R$ such that $\kappa(s)=0$, $s\leq 0$, and $\kappa(s)=1$, $s\geq 1$. Denote $\kappa_{\eps}(s):=\kappa\left( \frac{s}{ \eps } \right)$, $s \in \R $, and consider the following SDE 
\begin{equation} 
  \label{equ_approximating_equation_for_srbm}
  \begin{split}
    dx_{\eps}(t)&= \lambda \left(1-\kappa_{\eps}^2(x_{\eps}(t))\right)dt+\kappa_{\eps}(x_{\eps}(t))\sigma dw(t), \quad t\geq 0,\\
    x_{\eps}(0)&= x^0.
  \end{split}
\end{equation}
This SDE has a unique strong solution for every $\eps>0$.

{\it Step II. Tightness in an appropriate space.} 
Consider the following processes
\begin{align*}
  a_{\eps}(t):&=\lambda\int_{ 0 }^{ t }\left( 1-\kappa_{\eps}^2(x_{\eps}(s)) \right)ds, \quad t\geq 0,  \\
  \eta_{\eps}(t):&= \int_{ 0 }^{ t }\sigma \kappa_{\eps}(x_{\eps}(s))dw(s), \quad t\geq 0,
\end{align*}
and 
\[
  [\eta_{\eps}]_t= \int_{ 0 }^{ t }\sigma^2 \kappa^2_{\eps}(x_{\eps}(s))ds , \quad t\geq 0,
\]
where $[\eta_{\eps}]$ is the quadratic variation of the martingale $\eta_{\eps}$.

By the uniform boundedness of the coefficients of SDE~\eqref{equ_approximating_equation_for_srbm}, one can show that the family $\{ (x_{\eps}, a_{\eps}, \eta_{\eps}, [\eta_{\eps}]),\ \eps>0 \}$ is tight in $\left(\Cf[0,\infty)\right)^4$. By Prokhorov's theorem, one can choose a subsequence $(x_m,a_m,\eta_m,[\eta_m]):=(x_{\eps_m},a_{\eps_m},\eta_{\eps_m}, [\eta_{\eps_m}])$, $m\geq 1$, which converges to $(x,a,\eta,\rho)$ in $\left( \Cf[0,\infty) \right)^4$ in distribution as $m\to\infty$. By the Skorokhod representation theorem, we may assume that $(x_m,a_m,\eta_m,[\eta_m]) \to (x,a,\eta,\rho)$ a.s. as $m\to\infty$.
 
 {\it Step III. Properties of the limit process.} One can see that for every $T>0$ there exist a random element $\dot{\rho}$ in $\L_2[0,T]$ and a subsequence $N$ such that
\begin{equation} 
  \label{equ_convergence_of_sig_for_sr_bm}
  \sigma^2\kappa_m^2(x_m) \to \dot{\rho} \quad \mbox{in the weak topology of}\ \ \L_2[0,T]
\end{equation}
along $N$, for every $t \in[0,T]$
\begin{equation} 
  \label{equ_equalities_of_sr_bm}
  x(t)= x^0+a(t)+\eta(t), \quad \rho(t)= \int_{ 0 }^{ t }\dot{\rho}(s)ds, \quad a(t)=\lambda\left( t- \frac{1}{ \sigma^2 }\rho(t)\right),
\end{equation}
and $\eta$ is a continuous square-integrable martingale with quadratic variation $\rho$. We may assume that $N=\N$.

{\it Step IV. Identification of quadratic variation and drift.} Because of the discontinuity of the coefficients of equation~\eqref{equ_sticky_reflected_brownian_motion}, we cannot conclude directly that $\rho(t)=\int_{ 0 }^{ t }\sigma^2 \I_{\left\{ x(s)>0 \right\}}ds $ and $a(t)=\lambda\int_{ 0 }^{ t } \I_{\left\{ x(s)=0 \right\}}ds $, $t \in[0,T]$, which would imply that $x$ is a weak solution to SDE~\eqref{equ_sticky_reflected_brownian_motion}. We propose to overcome this problem as follows. Let us use the following facts:
\begin{enumerate}
  \item[a)] if $x(t)$, $t\geq 0$, is a continuous non-negative semimartingale with quadratic variation 
    \[
      [x]_t=\int_{ 0 }^{ t } \sigma^2(s)ds, \quad t\geq 0, 
    \]
    then almost surely
    \[
      [x]_t=\int_{ 0 }^{ t } \sigma^2(s)\I_{\left\{ x(s)>0 \right\}}ds, \quad t\geq 0;\footnote{see also Lemma~\ref{lem_quadratic_variation_of_rv_semimartingale}}
    \]

  \item[b)] if $s_m \to s$ in $\R $, then $\kappa_m^2(s_m)\I_{(0,+\infty)}(s) \to \I_{(0,+\infty)}(s)$ in $\R $ as $m\to\infty$.
\end{enumerate}
So, using~\eqref{equ_convergence_of_sig_for_sr_bm}, a), b) and the dominated convergence theorem, we get almost surely 
\begin{align*}
  \rho(t)&= \int_{ 0 }^{ t } \dot{\rho}(s)ds=\int_{ 0 }^{ t } \dot{\rho}(s)\I_{\left\{ x(s)>0 \right\}}ds\\
  &= \lim_{ m\to\infty }\int_{ 0 }^{ t } \sigma^2 \kappa_m^2(x_m(s))\I_{\left\{ x(s)>0 \right\}}ds =\int_{ 0 }^{ t }\sigma^2 \I_{\left\{ x(s)>0 \right\}}ds, \quad t \in[0,T]. 
\end{align*}
Hence,~\eqref{equ_equalities_of_sr_bm} implies 
\[
  a(t)=\lambda\left( t- \frac{1}{ \sigma^2 }\rho(t) \right)=\lambda \int_{ 0 }^{ t } \I_{\left\{ x(s)=0 \right\}}ds, \quad t \in[0,T].  
\]
Consequently, 
\[
  x(t)=x^0+\lambda \int_{ 0 }^{ t } \I_{\left\{ x(s)>0 \right\}}ds+\eta(t), \quad t\geq 0, 
\]
where $\eta$ is a continuous square-integrable martingale with quadratic variation 
\[
  [\eta]_t=\int_{ 0 }^{ t } \sigma^2 \I_{\left\{ x(s)>0 \right\}}ds, \quad t\geq 0,
\]
that means that $x$ is a weak solution to~\eqref{equ_sticky_reflected_brownian_motion}.
\vspace{3mm}

{\it Content of the paper.}
To show the existence of a weak solution to SPDE~\eqref{equ_main_equation}-\eqref{equ_initial_condition}, we will follow the argument above. Step I will be done in Section~\ref{sub:existence}. More precisely, we will construct a family of processes which will approximate a solution to SPDE~\eqref{equ_main_equation}-\eqref{equ_initial_condition}. The approximating sequence is similar to one considered in~\cite{Funaki:1983}. Section~\ref{sub:tightness} is devoted to the tightness that is Step II of our argument. Step III is made in Section~\ref{sub:martingale_problem_for_limit_points_of_the_discrete_approximation}, where we show that the limit process satisfies equalities similar to~\eqref{equ_equalities_of_sr_bm} (see Proposition~\ref{pro_subsequence_of_xn} there). An analog of property a) above is stated for some infinite-dimensional semimartingales in Theorem~\ref{the_properties_of_quadratic_variation_of_l2_valued_processes} in Section~\ref{sub:a_property_of_quadratic_variation_of_some_semi_martingales}. The proof of the existence theorem is given in Section~\ref{sub:proof_of_theorem}, where we use the approach described in Step~IV. Auxiliary statements are proved in the appendix.

\subsection{Preliminaries}%
\label{sub:preliminaries}

We will denote the inner product and the corresponding norm in a Hilbert space $H$ by $\langle \cdot , \cdot \rangle_H$ and $\|\cdot\|_H$, respectively.


For an essentially bounded function $\psi \in L_{\infty}$ we define the multiplication operator $\mo{\psi}$ on $\L_2$ as follows
\[
  \left(\mo{\psi}h\right)(u)=\psi(u)h(u),\quad u \in [0,1],\ \ h \in \L_2.
\]

Let $A$ be an operator on $\L_2$ and $\varphi_1,\dots,\varphi_n \in \L_2$ such that the product $\varphi_1 \dots \varphi_n$ belongs to $\L_2$. To ease notation, we will always write $A \varphi_1\dots \varphi_n$ for $A\left( \varphi_1\dots \varphi_n \right)$.

Denote the space of Hilbert-Schmidt operators on $\L_2$ by $\HS$. Remark that $\HS$ furnished with the inner product 
\[
  \left\langle A,B \right\rangle_{\HS}=\sum_{ n=1 }^{ \infty } \langle Ae_k , Be_k \rangle,\quad A,B \in \HS,
\]
is a Hilbert space, where the norm does not depend on the choice of basis in $\L_2$. The family of operators $\left\{e_k \odot e_l,\ k,l\geq 1\right\}$ form a basis in $\HS$. Here, for every $\varphi,\psi \in \L_2$ $\varphi \odot \psi$ denotes the operator on $\L_2$ defined as $(\varphi\odot \psi)g= \langle g , \psi \rangle \varphi$, \ $g \in \L_2$. 

We will consider the set $\R^{n \times n}$ of all $n \times n$-matrices with real enters as a Hilbert space with the Hilbert-Schmidt inner product $\langle A , B \rangle_{\R^{n \times n}}=\sum_{ k,l=1 }^{ n } A_{k,l}B_{k,l}$.

The indicator function will be defined as usually
\[
\I_S(x)=
\begin{cases}
  1, & \mbox{ if } x \in S, \\
  0, & \mbox{ if } x \not\in S.
\end{cases}
\] 
If $\phi:E_1 \to E_2$ is a function and $S$ is a subset of $E_2$, then $\I_{\left\{ \phi \in S \right\}}$ will denote the function $x \mapsto \I_{S}(\phi(x))$ from $E_1$ to $E_2$. 

Given a Hilbert space $H$, we write $H^T:=\L_2\left([0,T], H\right)$ for the class of all Bochner integrable functions $\Phi:[0,T] \to H$ with
\[
  \|\Phi\|_{H,T}=\left(\int_{ 0 }^{ T } \|\Phi_s\|_H^2ds\right)^{ \frac{1}{ 2 }}<\infty.
\]
One can show that the space $H^T$ equipped with the inner product  
\[
  \langle \Phi , \Psi \rangle_{H,T}=\int_{ 0 }^{ T } \langle \Phi_s , \Psi_s \rangle_Hds, \quad \Phi,\Psi \in H^T,
\]
is a Hilbert space.

Considering a sequence $Z^n$ in $\LL{ H}$, we will say that $Z^n \to Z$ a.e. as $n\to\infty$, if $\leb_T\{t \in[0,T]: Z^n_t \not\to Z_t \ \mbox{in}\ H,\ n\to\infty\}=0$, where $\leb_T$ denotes the Lebesgue measure on $[0,T]$.

Let $L \in \LL{ \HS}$, $Z \in \LL{ \L_2}$, and $S$ be a Borel measurable subset of $\R $. It is easily seen that $L_t\mo{\I_{\{Z_t \in S\}}}$, $t \in[0,T]$, where $L_t\mo{\I_{\{Z_t \in S\}}}$ is the composition of two operators, is well-defined and belongs to $L \in \LL{ \HS}$. We will denote such a function shortly by $L_{\cdot }\mo{\I_{\left\{ Z_{\cdot } \in S \right\}}}$.

Let $I$ be equal to $[0,T]$, $[0,\infty)$ or $[0,T] \times [0,1]$. The space of all continuous functions from $I$ to a Polish space $E$ with the topology of uniform convergence on compact sets is denoted by $\Cf\left(I,E\right)$. If $I=[0,T]$ or $[0,\infty)$, and $E=\R$, then we simply write $\Cf[0,T]$ or $\Cf[0,\infty)$ instead of $\Cf\left(I,\R\right)$.

We will denote the right continuous complete filtration generated by continuous processes $\xi_1(t)$, $t \in I$, \dots, $\xi_n(t)$, $t \in I$, by $(\F_t^{\xi_1,\dots,\xi_n})_{t \in I}$. Remark that such a filtration exists by Lemma~7.8~\cite{Kallenberg:2002}.

\section{Finite sticky reflected particle system}%
\label{sec:finite_sticky_reflected_particle_system}

In this section, we construct a sequence of random processes which will be used for the approximation of a solution to SPDE~\eqref{equ_main_equation}-\eqref{equ_initial_condition}. 

Let $n \geq 1$ be fixed. We set $\pi_k^n=\I_{\left[ \frac{ k-1 }{ n }, \frac{k}{ n } \right)}$, $k \in [n]:=\{ 1,\dots,n \}$. Let $W_t$, $t \geq 0$, be a cylindrical Wiener process in $\L_2$. Define the following Wiener processes on $\R $ as follows
\[
  w_k^n(t):=\sqrt{ n }\int_{ 0 }^{ t } \langle \pi_k^n , QdW_s \rangle, \quad t \geq 0,\ \ k \in [n], 
\]
and note that their joint quadratic variation 
\[
  \left[ w_k^n,w_l^n \right]_t=n\langle Q \pi_k^n , Q \pi_l^n \rangle t=:q^n_{k,l}t, \quad t \geq 0.
\]
Let also $\lambda_k^n:=n\langle \lambda , \pi_k^n \rangle \I_{\left\{ q_{k,k}^n>0 \right\}}$\footnote{We add the indicator $\I_{\left\{ q_{k,k}^n>0 \right\}}$ into the definition of $\lambda_k^n$ because we need the additional condition that $\lambda_k^n=0$ if $q_{k,k}^n=0$ for the existence of solution to SDE~\eqref{equ_system_of_sde}} and $g_k^n:=n \langle g , \pi_k^n \rangle$, \ $k \in [n]$.

Consider the following SDE 
\begin{equation} 
  \label{equ_system_of_sde}
  \begin{split}
  dx_k^n(t)&=  \frac{1}{ 2 }\Delta^n x^n_k(t)dt+\lambda_k^n \I_{\left\{ x_k^n(t)=0 \right\}}dt\\
  &+f(x_k^n(t))dt+\sqrt{ n }\I_{\left\{ x_k^n(t)>0 \right\}}dw_k^n(t),\quad k \in [n],
\end{split}
\end{equation}
satisfying the initial condition 
\begin{equation} 
  \label{equ_initial_condition_for_sde}
  x_k^n(0)=g_k^n, \quad k \in [n], 
\end{equation}
where $\Delta^nx^n_k=\left( \Delta^nx^n \right)_k=n^2\left(x_{k+1}^n+x_{k-1}^n-2x_k^n\right) $ and 
\begin{equation} 
  \label{equ_boundary_conditions_for_sde}
  x_0^n(t)=\nd x_1^n(t),\ \ x_{n+1}^n(t)=\nd x_n^n(t), \quad t\geq 0.
\end{equation}

We will construct a solution to SPDE~\eqref{equ_main_equation}-\eqref{equ_initial_condition} as a weak limit in $\Cf\left([0,\infty),\Cf[0,1]\right)$ of processes 
\begin{equation} 
  \label{equ_process_tilda_x}
  \tilde{X}^n_t(u)=(un-k+1)x_k^n(t)+(k-nu)x_{k-1}^n(t),  
\end{equation}
$t \in [0,T]$, $u \in \pi^n_k$, $k \in [n]$. Remark that equation~\eqref{equ_system_of_sde} has discontinuous coefficients. So the classical theory of SDE cannot be applied in our case. The existence of the solution will follow from Theorem~\ref{the_existence_of_solutions_to_system_of_sde} which we state below.

\subsection{SDE for sticky-reflected particle system}%
\label{sub:existence}

The aim of this section is to prove the existence of solutions to~\eqref{equ_system_of_sde},~\eqref{equ_initial_condition_for_sde}. We will formulate the problem in slightly general form. So, let $n \in \N$ and $g_k$, $\lambda_k$, $k \in [n]$, be non-negative numbers. We also consider a family of Brownian motions $w_k(t)$, $t\geq 0$, $k \in [n]$, (with respect to the same filtration) with joint quadratic variation
\[
  [w_k,w_l]_t=q_{k,l}t,\quad t\geq 0.
\]
Let as before $f:[0,\infty) \to [0,\infty)$ be a continuous function with with linear growth and $f(0)=0$. 
Consider the following SDE.
\begin{equation} 
  \label{equ_system_of_sde_general_case}
  \begin{split} %
    dy_k(t)&=  \frac{1}{ 2 }\Delta^n y_k(t)dt+\lambda_k \I_{\left\{ y_k(t)=0 \right\}}dt\\
    &+f(y_k(t))dt+\I_{\left\{ y_k(t)>0 \right\}}dw_k(t),\quad k \in [n],
  \end{split}
\end{equation}
with initial condition
\begin{equation} 
  \label{equ_initial_boundary_conditions_general_case}
  y_k(0)=g_k,\quad k \in [n],
\end{equation}
and the boundary conditions 
\begin{equation} 
  \label{equ_boundaty_condition_general_case}
  y_0(t)=\alpha_0y_1(t),\quad y_{n+1}(t)=\alpha_0y_n(t),\quad t\geq 0.
\end{equation}

\begin{theorem} 
  \label{the_existence_of_solutions_to_system_of_sde}
  Let $q_{k,k}=0$ imply $\lambda_k=0$ for every $k \in [n]$. Then there exists a family of non-negative (real-valued) continuous processes $y_k(t)$, $t \geq 0$, $k \in [n]$, in $\R $ which is a weak (martingale) solution to~\eqref{equ_system_of_sde},~\eqref{equ_initial_condition_for_sde}, that is, $y_k(0)=g_k$,  for each $k \in [n]$
  \begin{align*}
    \cN_k(t)&:=y_k(t)-g_k- \frac{1}{ 2 }\int_{ 0 }^{ t }\Delta^ny_k(s)ds\\
    &-\lambda_k \int_{ 0 }^{ t } \I_{\left\{ y_k(s)=0 \right\}}ds -\int_{ 0 }^{ t } f(y_k(s))ds, \quad t\geq 0, 
  \end{align*}
  is an $(\F^y_t)$-martingale, and the joint quadratic variation of $\cN_k$ and $\cN_l$, $k,l \in [n]$, equals 
  \[
    \left[ \cN_k,\cN_l \right]_t=q_{k,l}\int_{ 0 }^{ t } \I_{\left\{ y_k(s)>0 \right\}}\I_{\left\{ y_l(s)>0 \right\}} ds, \quad t \geq 0.
  \]
\end{theorem}

We are going to construct a solution to the SDE approximating the coefficients by Lipschitz continuous functions and using the method described in Section~\ref{sub:the_main_idea_of_passing_to_the_limit}.

Let us take a non-decreasing function $\kappa \in \Cf^1(\R )$ such that $\kappa(x)=0$ for $x\leq 0$, and $\kappa(x)=1$ for $x\geq 1$. Let also $\theta \in \Cf^1(\R )$ be a non-negative function with $\supp \theta \in [-1,1]$ and $\int_{ -\infty }^{ +\infty } \theta(x)dx=1 $.  For every $\eps>0$ we introduce the functions $\kappa_{\eps}(x)=\kappa\left( \frac{x}{ \eps } \right)$, $x \in \R $, and $\theta_{\eps}(x)= \frac{1}{ \eps }\theta\left( \frac{x}{ \eps } \right)$, $x \in \R $. Setting $f_{\eps}(x)=\int_{ 0 }^{ +\infty } \theta_{\eps}(x-y)f(y)dy $, $x \in \R $, we consider the following SDE
\begin{equation} 
  \label{equ_approximation_equation_for_system_of_sde}
    \begin{split}
    dy_k^{\eps}(t)&=  \frac{1}{ 2 }\Delta^ny_k^{\eps}(t)dt+\lambda_k \left(1-\kappa_{\eps}^2(y_k^{\eps}(t))\right)dt\\
    &+f_{\eps}(y_k^{\eps}(t))dt+\kappa_{\eps}(y_k^{\eps}(t))dw_k(t),\\
    y_k^{\eps}(0)&= g_k, \quad k \in [n].
  \end{split}
\end{equation}
Since equation~\eqref{equ_approximation_equation_for_system_of_sde} has locally Lipschitz continuous coefficients with linear growth, it has a unique strong solution. 

Our goal is to to show that the sequence $\left\{y^{\eps}=( y^{\eps}_k )_{k=1}^n\right\}_{\eps>0}$ has a subsequence which converges in distribution to a week solution to~\eqref{equ_system_of_sde}.  We denote for every $k \in [n]$
\[
  a^{\eps}_k(t)= \lambda_k\int_{ 0 }^{ t } \left( 1-\kappa_{\eps}^2(y_k^{\eps}(s)) \right)ds,\quad t\geq 0,
\]
and 
\[
  \eta^{\eps}_k(t)= \int_{ 0 }^{ t } \kappa_{\eps}(y_k^{\eps}(s))dw_k(s), \quad t \geq 0.
\]
Set $a^{\eps}=(a^{\eps}_k)_{k=1}^n$ and $\eta^{\eps}=( \eta^{\eps}_k )_{k=1}^n$.

The quadratic variation $\left[ \eta^{\eps} \right]_t$, $t\geq 0$, of the $\R^n $-valued martingale $\eta^{\eps}$ takes values in the space of non-negative definite $n \times n$-matrices with entries   
\[
  \left[ \eta^{\eps}_k,\eta^{\eps}_l \right]_t=\int_{ 0 }^{ t } \sigma^{\eps}_{k,l}(s)ds, 
\]
where $\sigma^{\eps}_{k,l}(s)=\kappa_{\eps}(y^{\eps}_k(s))\kappa_{\eps}(y^{\eps}_l(s))q_{k,l}$.

\begin{remark} 
  \label{rem_connection_between_a_and_sigma}
  According to the choice of the approximating sequence for $a$, the equality
  \[
    a_k^{\eps}(t)=\lambda_k\left( t- \frac{1}{ q_{k,k} } \left[\eta_k^{\eps}\right]_t \right), \quad t\geq 0,
  \]
 holds for every $k \in [n]$ satisfying $q_{k,k}>0$.
\end{remark}

Consider the following metric space $\W_{\R^n}:=\left(\Cf\left( [0,\infty),\R^n  \right)\right)^3 \times \Cf\left( [0,\infty),\R^{n \times n} \right)$.

\begin{lemma} 
  \label{lem_tightness_of_y_eps}
  The family $\left\{\left(y^{\eps_m},a^{\eps_m}, \eta^{\eps_m},\left[ \eta^{\eps_m} \right]\right),\ m\geq 1\right\}$ is tight in $\W_{\R^n }$, where $\eps_m$, $m\geq 1$, is any sequence convergent to zero. 
\end{lemma}

\begin{proof} %
  In order to prove the statement, it is sufficient to show that each family of coordinate processes of $\left( y^{\eps_m},a^{\eps_m},\eta^{\eps_m},\left[ \eta^{\eps_m} \right] \right)$, $m\geq 1$, is tight in the corresponding space. We will only show the tightness for $\left\{y^{\eps_m},\ m\geq 1\right\}$. The tightness for other families can be obtained similarly.

  According to the Aldous tightness criterion~\cite[Theorem 1]{Aldous:1978}, it is enough to show that for every $T>0$, any family of stopping times $\tau_m$, $m\geq 1$, bounded by $T$ and any sequence $\delta_m$ decreasing to zero
  \[
    y^{\eps_m}(\tau_m+\delta_m)-y^{\eps_m}(\tau_m) \to 0\quad \mbox{in probability as }m \to \infty,
  \]
  and $\{y^{\eps_m}(t),\ m\geq 1\}$ is tight in $\R^n $ for each $t \in[0,T]$. 
  
  The conditions above trivially follow from the convergence
  \[
    \E{\left\|y^{\eps_m}(\tau_m+\delta_m)-y^{\eps_m}(\tau_m) \right\|_{\R^n }^2 }\to 0\quad \mbox{as } m \to \infty,
  \]
  and the uniform boundedness of $\E{ \left\| y^{\eps_m}(t) \right\|_{\R^n }^2 }$ in $m\geq 1$ for every $t \in[0,T]$.

  Using the fact that there exists a constant $C>0$ such that $|f_{\eps_m}(x)|\leq C(1+|x|)$, $x \in \R $, $m\geq 1$, the inequality 
  \begin{align*}
    \langle y^{\eps_m}(t) , \Delta^ny^{\eps_m}(t) \rangle&=  -\sum_{ k=1 }^{ n-1 } (y^{\eps_m}_{k+1}(t)-y^{\eps_m}_k(t))^2\\
    &-\nd (y^{\eps_m}_1(t)+y^{\eps_m}_n(t))\leq 2 \eps_m\nd
  \end{align*}
  for all $t \in[0,T]$, the \Ito formula and Gronwall's lemma, one can check that for every $p\geq 1$ there exists a constant $C_{p,T,n}$, depending on $p$, $T$ and $n$, such that 
\begin{equation} 
  \label{equ_estimate_of_expectation_of_y_eps}
  \E {\|y^{\eps_m}(t)\|^{2p}_{\R^n }}\leq C_{p,T,n}, \quad t \in[0,T].
\end{equation}

Next, by the \Ito formula and the optional sampling Theorem~7.12~\cite{Kallenberg:2002}, we have
  \begin{equation} 
  \label{equ_estimate_of_increment_of_y_eps}
    \begin{split}
    \e\big[\|y^{\eps_m}(\tau_m+\delta_m) &-y^{\eps_m}(\tau_m)\|^2_{\R^n }\big]\leq \E{\int_{ \tau_m }^{ \tau_m+\delta_m }\langle y^{\eps_m}(r) , \Delta^ny^{\eps_m}(r) \rangle_{\R^n }dr}\\
  &+2 \E{\int_{ \tau_m }^{ \tau_m+\delta_m } \langle y(r) , \lambda\left( 1-\kappa^2_{\eps_m}(y_{\cdot }(r)) \right) \rangle_{\R^n } dr}\\
  &+ 2\E{\int_{ \tau_m }^{ \tau_m+\delta_m } \langle y^{\eps_m}(r) , f_{\eps_m}(y^{\eps_m}_{\cdot }(r)) \rangle_{\R^n } dr}\\
  &+\E{\int_{ \tau_m }^{ \tau_m+\delta_m } \sum_{ k=1 }^{ n } \kappa_{\eps_m}^2(y^{\eps_m}_{k }(r))q_{k,k}dr}
    \end{split}
  \end{equation}
  for all $m\geq 1$. Using H\"older's inequality, and estimate~\eqref{equ_estimate_of_expectation_of_y_eps} one can conclude that 
\[
  \E{ \left\|y^{\eps_m}(\tau_m+\delta_m)-y^{\eps_m}(\tau_m)\right\|^2_{\R^n } }\to 0 \quad \mbox{as }m \to \infty.
\]
This completes the proof of the lemma.
\end{proof}

By Lemma~\ref{lem_tightness_of_y_eps} and Prokhorov's theorem, there exists a sequence $\{\eps_m\}_{m\geq 1}$ converging to zero such that the sequence $\by^{\eps_m}:=\left( y^{\eps_m},a^{\eps_m},\eta^{\eps_m},[\eta^{\eps_m}] \right)$ converges to a random element $\by:=\left( y,a,\eta,\rho\right)$ in $\W_{\R^n }$ in distriburion. By the Skorokhod representation Theorem~3.1.8~\cite{Ethier:1986}, one can choose a probability space $(\tilde{\Omega},\tilde{\F},\tilde{\p})$ and determine there a family of random elements $\tilde{\by}$, $\tilde{\by}^{\eps_m}$, $m\geq 1$ taking values in $\W_{\R^n }$ such that $\law \tilde{\by}=\law \by$, $\law \tilde{\by}^{\eps_m}=\law \by^{\eps_m}$, $m\geq 1$, and $\tilde{\by}^{\eps_m} \to \by$ in $\W_{\R^n }$ a.s. So, without loss of generality, we will assume that $\by^{\eps_m} \to \by$ in $\W_{\R^n }$ a.s. as $m \to \infty$. Since the sequence $\{\eps_m\}_{m\geq 1}$ will be fixed to the end of this section, we will write $m$ instead of $\eps_m$ in order to simplify the notation.

Let $y=(y_k)_{k=1}^n,a=(a_k)_{k=1}^n,\eta=( \eta_k )_{k=1}^n,\rho=(\rho_{k,l})_{k,l=1}^n$.

\begin{lemma} 
  \label{lem_basic_properties_of_y}
  \begin{enumerate}
    \item[(i)] The coordinate processes $y_k(t)$, $t \geq 0$, $k \in [n]$, of $y$ are non-negative and 
      \[
	y_k(t)=g_k+ \frac{1}{ 2 }\int_{ 0 }^{ t } \Delta^ny_k(s)ds+a_k(t)+\int_{ 0 }^{ t } f(y_k(s))ds+\eta_k(t), \quad t\geq 0,\quad k \in [n].
      \]

    \item [(ii)] For every $k \in [n]$ such that $q_{k,k}>0$ one has 
      \[
	a_k=\lambda_k\left( t- \frac{1}{ q_{k,k}} \rho_{k,k}  \right).
      \]

    \item [(iii)] For every $k \in [n]$ and $T>0$ there exists a random element $\dot{a}_k$ in $\L_2([0,T],\R )$ such that almost surely
      \[
	a_k(t)=\int_{ 0 }^{ t } \dot{a}_k(s)ds, \quad t \in [0,T]. 
      \]

    \item [(iv)] For every $k,l \in [n]$ and $T>0$ there exists a random element $\dot{\rho}_{k,l}$ in $\L_2\left([0,T], \R \right)$ such that almost surely 
      \[
	\rho_{k,l}(t)=\int_{ 0 }^{ t } \dot{\rho}_{k,l}(s)ds,\quad t \in [0,T].
      \]

    \item [(v)] For every $k \in [n]$ the process $\eta_k(t)$, $t \geq 0$, is a continuous square-integrable $(\F^{\eta})$-martingale, and the joint quadratic variation of $\eta_k$ and $\eta_l$, $k,l \in [n]$, equals 
      \[
	\left[ \eta_k,\eta_l \right]_t=\rho_{k,l}(t),\quad t \geq 0. 
      \]
  \end{enumerate}
\end{lemma}

\begin{proof} %
  We remark that for every $k \in [n]$
  \[
    \P{\forall t \geq 0\ \  f_m(y^m_k(t))\to f(y_k(t))\ \ \mbox{as}\ \ m\to\infty}=1,
  \]
  and for every $m\geq 1$ and $k \in [n]$ almost surely 
  \[
    y^m_k(t)=g_k+ \frac{1}{ 2 }\int_{ 0 }^{ t } \Delta^ny^m_k(s)ds+a^m_k(t)+\int_{ 0 }^{ t } f_m(y^m_k(s))ds+\eta^m_k(t), \quad t\geq 0.
  \]
  Passing to the limit and using the dominated convergence theorem, we obtain the equality (i).

  The equality in (ii) follows from Remark~\ref{rem_connection_between_a_and_sigma} and the convergence in distribution of $(a^m_k,\left[ \eta_k^m \right])$ to $(a_k,\rho_{k,k})$ in $\left(\Cf\left( [0,+\infty),\R  \right)\right)^2$.

  We next prove (iii). Let $T>0$ be fixed. Denote the ball in $\L_2\left([0,T], \R \right)$ with center 0 and radius $r>0$ by $B_r^T$ and furnish it with the weak topology of the space $\L_2\left([0,T], \R\right)$, i.e. a sequence $\{h_m\}_{m\geq 1}$ converges to $h$ in $B_r^T$ if $\langle h_m , b \rangle_{\R,T} \to \langle h , b \rangle_{\R ,T}$ for all $b \in B_r^T$. By Alaoglu's Theorem~V.4.2~\cite{Dunford:1988} and Theorem~V.5.1~ibid, $B_r^T$ is a compact metric space. 
  
  We fix $k \in [n]$ and take $r:=\lambda_k\sqrt{T}$, 
  \[
    \dot{a}_k^m(t):=\lambda_k\left( 1-\kappa^2_{m}(y^m_k(t)) \right),\quad t \in [0,T].
  \]
  Then $\dot{a}_k^m$ is a random element in $B_r^T$ for every $m\geq 1$. By the compactness of $B_r^T$, the family $\{ \dot{a}_k^m,\ m\geq 1 \}$ is tight in $B_r^T$. Consequently, Prokhorov's theorem implies the existance of a subsequence $N \subset \N$ such that $\dot{a}_k^m\to\tilde{a}_k$ in $B_r^T$ in distribution along $N$. In particular, for every family $t_1,\dots,t_l \in [0,T]$ and numbers $c_1,\dots,c_l \in \R $,
  \begin{align*}
    \sum_{ i=1 }^{ l } c_i \int_{ 0 }^{ t_i } \dot{a}_k^m(s)ds &=  \int_{ 0 }^{ T }\left( \sum_{ i=1 }^{ l } c_i \I_{[0,t_i]}(s) \right) \dot{a}_k^m(s)ds \\
    & \to \int_{ 0 }^{ T } \left( \sum_{ i=1 }^{ l } c_i\I_{[0,t_i]}(s) \right) \tilde{a}_k(s)ds=\sum_{ i=1 }^{ l } c_i \int_{ 0 }^{ t_i } \tilde{a}_k(s)ds  
  \end{align*}
  in $\R $ in distribution along $N$. Since the family of functions 
  \[
  \left\{ x \mapsto h\left( \sum_{ l=1 }^{ l } c_ix_i \right),\ x=( x_i )_{i=1}^l \in \R^l :\ \ h\ \mbox{is continuous and bounded on}\ \R \right\}
\]
strongly separates points\footnote{see the definition on p. 113~\cite{Ethier:1986}}, Theorem~3.4.5~\cite{Ethier:1986} yields that
  \[
    \left( \int_{ 0 }^{ t_1 } \dot{a}_k^m(s)ds,\dots,\int_{ 0 }^{ t_l } \dot{a}_k^m(s)ds   \right) \to \left( \int_{ 0 }^{ t_1 } \tilde{a}_k(s)ds,\dots,\int_{ 0 }^{ t_l } \tilde{a}_k(s)ds   \right)
  \]
  in $\R^l $ in distribution along $N$. From the other hand side, 
  \[
    \left( \int_{ 0 }^{ t_1 } \dot{a}_k^m(s)ds,\dots,\int_{ 0 }^{ t_l } \dot{a}_k^m(s)ds   \right) \to \left( a_k(t_1),\dots,a_k(t_l)   \right)
  \]
  in $\R^l $ a.s. along $N$. This implies that 
  \begin{equation} 
  \label{equ_distribution_of_ak}
    \law a_k=\law \int_{ 0 }^{ \cdot  } \tilde{a}_k(s)ds .
  \end{equation}
  
  Let us show that there exists a random element $\dot{a}_k$ in $\L_2\left([0,T], \R\right)$ such that $a_k=\int_{ 0 }^{ \cdot  } \dot{a}_k(s)ds $ a.s. We define the map $\Phi: \L_2\left([0,T], \R\right)\to \Cf[0,T]$
  \[
    \Phi(h)(t)=\int_{ 0 }^{ t } h(s)ds,\quad t \in [0,T]. 
  \]
  Remark that $\Phi$ is a bijective map from $\L_2\left([0,T], \R \right)$ to its image 
  \[
    \Im\Phi=\left\{ \Phi(h):\ h \in \L_2\left([0,T], \R \right) \right\}.
  \]
  By the Kuratowski Theorem~A.10.5~\cite{Ethier:1986}, the set $\Im\Phi$ is Borel measurable in $\Cf[0,T]$ and the map $\Phi^{-1}$ is Borel measurable. By~\eqref{equ_distribution_of_ak}, $a_k \in \Im\Phi$ a.s. Thus, we can define $\dot{a}_k=\Phi^{-1}(a_k)$. This completes the proof of (iii).

  Similarly, one can prove (iv).

  Statement (v) follows from the fact that the limit of local martingales is a local martingale and the uniform boundedness of $\E{ \left(\eta_k^m(t)\right)^2 }$ in $m$. Indeed, for every $k,l \in [n]$ the processes $\eta_k^m$ and $\eta_k^m\eta_l^m-[\eta_k^m,\eta_l^m]$ are $(\F^{\eta^m}_t)$-martingales for all $m\geq 1$, and $\left(\eta^m,[\eta^m],\eta_k^m \eta_l^m-[\eta_k^m,\eta_l^m]\right) \to (\eta,\rho,\eta_k \eta_l-\rho_{k,l})$ a.s. as $m\to\infty$. Proposition~IX.1.17~\cite{Jacod:2003} implies that $\eta_k$ and $\eta_k \eta_l -\rho_{k,l}$ are $(\F^{(\eta,\rho)}_t)$-local martingales. Note that, by the Fisk approximation Theorem~17.17 \cite{Kallenberg:2002}, $\F_t^{(\eta,\rho)}=\F_t^{\eta}$, $t \geq 0$. Using the uniform boundedness of $\E{ \left( \eta_k^m(t) \right)^2 }$ in $m$ and Fatou's lemma, one can see that $\eta_k$ is a square-integrable $(\F_t^{\eta})$-martingale. This finishes the proof of the lemma.
\end{proof}
  
\begin{proposition} 
  \label{pro_identification_of_y}
  Let $\by(t)=(y(t),a(t),\eta(t),\rho(t))$, $t \geq 0$, be as in Lemma~\ref{lem_basic_properties_of_y}. Let additionally $\lambda_k=0$ if $q_{k,k}=0$, $k \in [n]$. Then 
  \begin{enumerate}
    \item[1)] for every $k,l \in [n]$ almost surely
	\[
	  \rho_{k,l}(t)=q_{k,l}\int_{ 0 }^{ t } \I_{\left\{ y_k(s)>0 \right\}}\I_{\left\{ y_l(s)>0 \right\}}ds, \quad t \geq 0.
	\]
    
      \item [2)] for every $k \in [n]$ almost surely
      \[
	a_k(t)=\lambda_k\int_{ 0 }^{ t } \I_{\left\{ y_k(s)=0 \right\}}ds, \quad t \geq 0.
      \]
  \end{enumerate}
\end{proposition}

\begin{proof} %
  We take the sequence $\left\{\by^m_n\right\}_{n\geq 1}$ as in the proof of Lemma~\ref{lem_basic_properties_of_y}. Again without loss of generality we may assume that it converges to $\by$ a.s. We first show that almost surely
  \[
    \rho_{k,l}(t)=q_{k,l}\int_{ 0 }^{ t } \I_{\left\{ y_k(s)>0 \right\}}\I_{\left\{ y_l(s)>0 \right\}}ds, \quad t \geq 0.
  \]
  Recall that almost surely
  \[
    [\eta^m_k,\eta^m_l]_t=\int_{ 0 }^{ t } \sigma^m_{k,l}(s)ds,\quad t \geq 0, 
  \]
  where $\sigma^m_{k,l}(s)=q_{k,l}\kappa_m(y_k^m(s))\kappa_m(y_l^m(s))$, and for each $T>0$, $k,l \in [n]$ there exist random elements $\dot{\rho}_{k,l}$ in $\L_2\left([0,T], \R \right)$ such that almost surely
  \[
    \rho_{k,l}(t)=\int_{ 0 }^{ t } \dot{\rho}_{k,l}(s)ds,\quad t \in[0,T], 
  \]
  by Lemma~\ref{lem_basic_properties_of_y}. 
  
  Let $T>0$, $k,l \in [n]$ be fixed. By the convergence of the sequence $[\eta^m_k,\eta^m_l]$, $m\geq 1$, to $\rho_{k,l}$ in $\Cf[0,T]$ a.s., the uniform boundedness of $\sigma_{k,l}^m$, and the density of $\spann\left\{\I_{[0,t]},\ t \in[0,T]\right\}$ in $\L_2\left([0,T], \R \right)$, one has that  
  \begin{equation} 
  \label{equ_convergence_of_sigma_to_dotrho}
    \P{ \sigma^m_{k,l} \to \dot{\rho}_{k,l}\ \mbox{in the weak topology of}\ \L_2\left([0,T], \R \right)\ \mbox{as}\ m\to\infty }=1.
  \end{equation}

  By Lemma~\ref{lem_basic_properties_of_y}, $y_k$ and $y_l$ are non-negative continuous semimartingales with quadratic variation $[y_k,y_l]_t=\rho_{k,l}(t)=\int_{ 0 }^{ t } \dot{\rho}_{k,l}(s)ds $, $t \in[0,T]$. Thus, Lemma~\ref{lem_quadratic_variation_of_rv_semimartingale} implies that a.s.
  \[
    \int_{ 0 }^{ t }\dot{\rho}_{k,l}(s)ds=\int_{ 0 }^{ t } \dot{\rho}_{k,l}(s)\I_{\left\{ y_k(s)>0 \right\}} \I_{\left\{ y_l(s)>0 \right\}}ds, \quad t \in[0,T].   
  \]
  The latter equality and~\eqref{equ_convergence_of_sigma_to_dotrho} yield that almost surely for every $t \in[0,T]$
  \begin{align*}
    \rho_{k,l}(t)&= \int_{ 0 }^{ t } \dot{\rho}_{k,l}(s)ds= \int_{ 0 }^{ t } \dot{\rho}_{k,l}(s)\I_{\left\{ y_k(s)>0 \right\}}\I_{\left\{ y_l(s)>0 \right\}}ds\\
    &=\lim_{ m\to\infty }\int_{ 0 }^{ t } \sigma_{k,l}^m \I_{\left\{ y_k(s)>0 \right\}}\I_{\left\{ y_l(s)>0 \right\}}ds \\
    &=  \lim_{ m\to\infty }\int_{ 0 }^{ t } q_{k,l}\kappa_m(y_k^m(s))\kappa_m(y_l^m(s))\I_{\left\{ y_k(s)>0 \right\}}\I_{\left\{ y_l(s)>0 \right\}}ds\\
    &= \int_{ 0 }^{ t } q_{k,l}\I_{\left\{ y_k(s)>0 \right\}}\I_{\left\{ y_l(s)>0 \right\}}ds,
  \end{align*}
  where we have used the convergence $\kappa_m(x_m)\I_{(0,+\infty)}(x) \to \I_{(0,+\infty)}(x)$ as $x_m \to x$ in $\R $ and the dominated convergence theorem. Hence, a.s.
  \[
    \rho_{k,l}(t)=\int_{ 0 }^{ t } q_{k,l}\I_{\left\{ y_k(s)>0 \right\}}\I_{\left\{ y_l(s)>0 \right\}}ds, \quad t\geq 0,
  \]
  and, consequently, according to Lemma~\ref{lem_basic_properties_of_y}~(ii), almost surely
  \[
    a_k(t)=\lambda_k\left( 1- \frac{1}{ q_{k,k} }\rho_{k,k}(t) \right)=\lambda_k \int_{ 0 }^{ t } \I_{\left\{ y_k(s)=0 \right\}}ds, \quad t\geq 0,
  \]
  for all $k \in [n]$ such that $q_{k,k}=0$. If $q_{k,k}=0$, then $\lambda_k=0$, by the assumption of Proposition~\ref{pro_identification_of_y}. Therefore, $a_k^m=0$ implies that $a_k=0$. This finishes the proof of the propostion.

\end{proof}

\begin{proof}[Proof of Theorem~\ref{the_existence_of_solutions_to_system_of_sde}] %
  The statement of the theorem directly follows from Lemma~\ref{lem_basic_properties_of_y} and Proposition~\ref{pro_identification_of_y}.
\end{proof}

\subsection{Tightness}%
\label{sub:tightness}

Let a family of non-negative continuous processes $\left\{ x_k^n(t),\ t\geq 0,\ k \in [n] \right\}$ be a weak solution to SDE~\eqref{equ_system_of_sde}-\eqref{equ_boundary_conditions_for_sde}, which exists according to Theorem~\ref{the_existence_of_solutions_to_system_of_sde}. Let the continuous process $\tilde{X}^n_t$, $t\geq 0$, taking values in $\Cf[0,1]$ be defined by~\eqref{equ_process_tilda_x}. We note that $\tilde{X}^n_t(u)\geq 0$ for all $u \in[0,1]$, $t\geq 0$ and $n\geq 1$. 

The aim of this section is to prove the tightness of the family $\left\{ \tilde{X}^n,\ n\geq 1 \right\}$ in $\Cf\left( [0,\infty),\Cf[0,1] \right)$. The similar problem was considered in~\cite[Section~2]{Funaki:1983}, where the author study the existence of solutions to an SPDE with Lipschitz continuous coefficients. The tightness argument there is based on properties of fundamental solution to the discrete analog of the heat equation and the fact that coefficients of the equation has at most linear growth. The Lipschitz continuity was not needed for the proof of the tightness. Since the proof in our case repeats the proof from~\cite{Funaki:1983}, we will point out only its main steps. The main statement of this sections reads as follows.

\begin{proposition} 
  \label{pro_tightness_of_tilde_x}
  The family of processes $\left\{ \tilde{X}^n,\ n\geq 1 \right\}$ is tight in $\Cf\left( [0,\infty),\Cf[0,1] \right)$.
\end{proposition}

For the proof of the proposition it is enough to show that the family $\left\{ \tilde{X}^n,\ n\geq 1 \right\}$ is tight in $\Cf\left( [0,T],\Cf[0,1] \right)=\Cf\left( [0,T] \times [0,1], \R  \right)$ for every $T>0$. So, we fix $T>0$, and use Corollary~16.9~\cite{Kallenberg:2002} which yields the tightness if $\left\{ \tilde{X}^n,\ n\geq 0 \right\}$ satisfies the following conditions:
\begin{enumerate}
  \item [1)] $\left\{ \tilde{X}^n_0(0),\ n\geq 1 \right\}$ is tight in $\R $;

\item [2)] there exist constants $\alpha,\beta,C>0$ such that 
\[
  \E{ |\tilde{X}^n_t(u)-\tilde{X}^n_s(v)|^{\alpha} }\leq C\left(|t-s|^{2+\beta}+|u-v|^{2+\beta}\right)
\]
for all $t,s \in [0,T]$, $u,v \in [0,1]$, and $n\geq 1$.
\end{enumerate}

The family $\left\{ \tilde{X}^n,\ n\geq 1 \right\}$ trivially satisfies the first condition because $\tilde{X}^n_0(0)=g_1^n$ is uniformly bounded in $n\geq 1$. In order to check the second condition, we first write equation~\eqref{equ_system_of_sde} in the integral form. Let $\{p^n_{k,l}(t),\ t\geq 0,\ k,l \in [n]\}$ be the fundamental solution of the system of ordinary differential equations\footnote{for more details about properties of the fundamental solution to the discrete analog of the heat equation see e.g.~\cite[Appendix~II]{Funaki:1983}} 
\[
  \frac{d}{ dt }p^n_{k,l}(t)= \frac{1}{ 2 }\Delta^n_{(k)}p^n_{k,l}(t), \quad t>0, \ k,l \in [n],
\]
with the initial condition 
\[
  p^n_{k,l}(0)=n \I_{\left\{ k=l \right\}}, \quad k,l \in [n],
\]
and the boundary conditions 
\[
  p^n_{0,l}(t)=\alpha_0p^n_{1,l}(t),\quad p^n_{n+1,l}(t)=\alpha_0p^n_{n,l}(t),\quad t\geq 0,\ l \in [n],
\]
where the operator $\Delta^n_{(k)}=\Delta^n$ is applied to the vector $(p_{k,l}(t))_{k=1}^n$ for every $l \in [n]$. Noting that $\left\{ \left\langle W_t , \sqrt{ n }\pi_k^n \right\rangle,\ t\geq 0,\ k \in [n] \right\}$ is a family of standard Brownian motions, it is easily seen that $\tilde{X}^n$ has the same distribution as the solution to the integral equation 
\begin{equation} 
  \label{equ_equation_in_integral_form}
  \begin{split}
    \tilde{X}^n_t(u)&=\int_{ 0 }^{ 1 } p^n(t,u,v)g(v)dv+\int_{ 0 }^{ t } \int_{ 0 }^{ 1 } p^n(t-s,u,v)\tilde{\lambda}^n(v)\I_{\left\{ \tilde{X}^n_s(\lceil v \rceil)=0 \right\}}dsdv\\
  &+ \int_{ 0 }^{ t } \int_{ 0 }^{ 1 }  p^n(t-s,u,v)f\left(\tilde{X}^n_s(\lceil v\rceil)\right)dsdv \\
  &+ \int_{ 0 }^{ t } \int_{ 0 }^{ 1 } p^n(t-s,u,v) \I_{\left\{ \tilde{X}^n_s(\lceil v\rceil)>0 \right\}}Q dW_sdu, \quad t\geq 0,\ u \in [0,1],  
  \end{split}
\end{equation}
where 
\begin{align*}
  p^n\left(t,u,v\right)&= (1-n(\lceil u\rceil-u))p^n_{k,n\lceil v\rceil}(t)+(\lceil u\rceil-u)p^n_{k,n\lceil v\rceil-1}(t),\quad t\geq 0,\\
  \tilde{\lambda}^n(v)&= \lambda(v)\I_{\left\{ q_{n\lceil v\rceil,n\lceil v\rceil}^n>0 \right\}}, \quad v \in [0,1],
\end{align*}
and $\lceil v\rceil = \lceil v\rceil^n:= \frac{l}{ n }$ for $v \in \pi_l^n$, $l \in [n]$. We will denote by $\tilde{X}^{n,i}_t(u)$ the $i$-th term of the right hand side of equation~\eqref{equ_equation_in_integral_form}.

\begin{lemma} 
  \label{lem_boundedness_of_the_expectation_of_tilde_x}
  For every $\gamma>0$ and $T>0$ there exists a constant $C>0$ such that 
  \[
    \E{ \left(\tilde{X}^n_t(u)\right)^{\gamma} }\leq C
  \]
  for all $t \in[0,T]$, $u \in[0,1]$ and $n\geq 1$.
\end{lemma}

\begin{lemma} 
  \label{lem_estimate_of_increments_of_tilde_x}
  For each $\gamma \in \N$ and $T>0$ there exists a constant $C>0$ such that 
  \[
    \E{ \left|\tilde{X}^{n,i}_{t_2}(u_2)-\tilde{X}^{n,i}_{t_1}(u_1)\right|^{2 \gamma} }\leq C\left( |t_2-t_1|^{\frac{ \gamma }{ 2 }}+|u_2-u_1|^{\frac{ \gamma }{ 2 }} \right)
  \]
  for every $t_1,t_2 \in[0,T]$, $u_1,u_2 \in[0,1]$, $n\geq 1$ and $i \in [4]$.
\end{lemma}

To prove lemmas~\ref{lem_boundedness_of_the_expectation_of_tilde_x} and~\ref{lem_estimate_of_increments_of_tilde_x}, one needs to repeat the proofs of lemmas~2.1 and~2.2 from~\cite{Funaki:1983} which are based on properties of the fundamental solution $p^n(t,u,v)$, $t \in[0,T]$, $u,v \in [0,T]$, and the fact that the coefficients of the equation has at most the linear growth. We omit the proof of those lemmas here. 

Proposition~\ref{pro_tightness_of_tilde_x} follows from Lemma~\ref{lem_estimate_of_increments_of_tilde_x}.

\begin{remark} 
  \label{rem_properties_of_a_limit_points_of_tilde_x}
  Let $\tilde{X}_t$, $t\geq 0$, be a limit point of the sequence $\left\{ \tilde{X}^n,\ n\geq 1 \right\}$ in $\Cf\left( [0,\infty),\Cf[0,1] \right)$, i.e $X$ is a limit in distribution of a subsequence of $\left\{ \tilde{X}^n,\ n\geq 1 \right\}$. Then the map $(t,u)\mapsto \tilde{X}_t(u)$ is a.s. locally H\"older continuous with exponent $\alpha \in (0,1/4)$, according to Lemma~\ref{lem_estimate_of_increments_of_tilde_x} and Corollary~16.9~\cite{Kallenberg:2002}. Moreover, Lemma~\ref{lem_boundedness_of_the_expectation_of_tilde_x} and Lemma~4.11~\cite{Kallenberg:2002} imply that for every $\gamma>0$ and $T>0$ there exists a constant $C=C(T,\gamma)$ such that
  \[
    \E{ \left(\tilde{X}_t(u)\right)^{\gamma} }\leq C,\quad t \in[0,T],\ u \in[0,1].
  \]
\end{remark}

\section{Passing to the limit}%
\label{sec:passing_to_the_limit}

The goal of the present section is to show that there exists a solution to SPDE~\eqref{equ_main_equation}-~\eqref{equ_initial_condition}. The solution will be constructed as a limit point of the family of processes $\left\{\tilde{X}^n,\ n\geq 1\right\}$ from Proposition~\ref{pro_tightness_of_tilde_x}, which exists by Prokhorov's theorem. Since the coefficients of the equation are discontinuous, we cannot pass to the limit directly. In the next section, we will show that there exists a subsequence of $\left\{\tilde{X}^n,\ n\geq 1\right\}$ whose weak limit in $\Cf\left( [0,\infty),\Cf[0,1] \right)$ is a heat semimartingale\footnote{we call continuous processes in $\L_2$ satisfying~\eqref{equ_heat_semimartingale} below a {\it heat semimartingales}}. After that we will prove an analog of the \Ito formula and state a property similar to one for usual $\R $-valued semimartingales, stated in Lemma~\ref{lem_quadratic_variation_of_rv_semimartingale}, for such heat semimartingales. Then, using the argument described in Section~\ref{sub:the_main_idea_of_passing_to_the_limit}, we show that $\tilde{X}$ solves equation~\eqref{equ_main_equation}-\eqref{equ_initial_condition}. In this section, $T>0$ will be fixed.

\subsection{Martingale problem for limit points of the discrete approximation}%
\label{sub:martingale_problem_for_limit_points_of_the_discrete_approximation}

We start from the introduction of a new metric space where we will consider the convergence. Denote 
\[
  r_0:=\left(1+\|\lambda\|+\left\|Q\right\|_{\HS}\right)\sqrt{ T },
\]
and consider the following balls
\begin{align*}
  \rB\left( \L_2 \right):&= \left\{f \in \LL{ \L_2}:\ \left\|f\right\|_{\L_2,T}\leq r_0\right\},\\
  \rB\left( \HS \right):&= \left\{L \in \LL{ \HS}:\ \left\|L\right\|_{\HS,T}\leq r_0\right\}
\end{align*}
in the Hilbert spaces $\LL{ \L_2}$ and $\LL{ \HS}$, respectively. We furniture this sets with the induced weak topologies. By Theorem~V.5.1~\cite{Dunford:1988}, those spaces are metrizable. Moreover, by Alaoglu's Theorem~V.4.2~\cite{Dunford:1988}, they are compact metric spaces.

For every $n\geq 1$ we take the family of processes $\left\{ x_k^n(t),\ t \in [0,T],\ k \in [n] \right\}$ that is a solution to SDE~\eqref{equ_system_of_sde}-~\eqref{equ_boundary_conditions_for_sde}. Let $\tilde{X}^n_t$, $t \in[0,T]$, be the continuous process in $\Cf[0,1]$ defined by~\eqref{equ_process_tilda_x}, that is, 
\[
  \tilde{X}^n_t(u)=(un-k+1)x_k^n(t)+(k-nu)x^n_{k-1}(t),\quad u \in [0,1],\ \ t \in[0,T].
\]
Let us also introduce the process 
\[
  X^n_t:= \sum_{ k=1 }^{ n } x_k^n(t)\pi_k^n,\quad t \in[0,T],
\]
where $\pi_k^n=\I_{\left\{ \left[ \frac{ k-1 }{ n }, \frac{k}{ n } \right) \right\}}$. Set
\begin{align*}
  \lambda^n:&=\sum_{ k=1 }^{ n } n \langle \lambda , \pi_k^n \rangle \I_{\left\{ q_{k,k}^n>0 \right\}}\pi_k^n \in \L_2,\\
  L_t^n:&=Q\mo{\I_{\left\{ X_t^n>0 \right\}}}\pr^n,\quad t \in[0,T],
\end{align*}
and 
\[
  \Gamma^n_t:=\left( L_t^n \right)^*L_t^n=\pr^n \mo{\I_{\{X^n_t>0\}}}Q^2\mo{\I_{\left\{ X_t^n>0 \right\}}}\pr^n,\quad t \in[0,T].
\]

We can trivially estimate $\|\lambda^n\|\leq \|\lambda\|$ and 
\[
  \|\Gamma^n_t\|_{\HS}\leq \left\|Q\mo{\I_{\left\{ X_t^n>0 \right\}}}\pr^n\right\|_{\HS}^2\leq \left\|Q\right\|_{\HS}^2,\quad  t \in[0,T].
\]
The latter inequality follows from Lemma~\ref{lem_norm_of_composition_of_hilbert_shmidt_operators}. Hence $\lambda^n \I_{\{X^n=0\}}$ and $\Gamma^n$ are random elements in $\LL{ \L_2}$ and $\LL{ \HS}$, respectively.
Let us consider the random element
\begin{equation} 
  \label{equ_vector_bx}
  \bX^n:=\left( \tilde{X}^n,X^n,\lambda^n \I_{\left\{ X^n=0 \right\}}, \I_{\left\{ X^n>0 \right\}},\Gamma^n \right),\quad n\geq 1,
\end{equation}
in the complete separable metric space
\[
  \W_{\L_2}=\Cf\left([0,T], \Cf[0,1] \right)\times \Cf\left([0,T], \L_2 \right)\times \rB\left(\L_2\right)^2 \times \rB\left(\HS\right).
\]
The following statement is the main result of this section.

\begin{proposition} 
  \label{pro_subsequence_of_xn}
  There exists a subsequence of $\left\{\bX^n,\ n\geq 1\right\}$ which converges to $\bX=\left( \tilde{X},X,a,\sigma,\Gamma \right)$ in $\W_{\L_2}$ in distribution. Moreover, the limit $\bX$ satisfies the following properties:
  \begin{enumerate}
    \item [(i)] $\tilde{X}_t=X_t$ in $\L_2$ for all $t \in[0,T]$ a.s. and, $a=\lambda(1-\sigma)$ in $\LL{ \L_2}$ a.s.;

    \item [(ii)] there exists a random element $L$ in $\LL{ \HS}$ such that 
      \[
	\P{ L^2=\Gamma\ \ \mbox{and}\ \ L\ \mbox{is self-adjoint a.e.} }=1,
      \]
     and
      \begin{equation} 
  \label{equ_boundedness_of_expectation_of_norm_of_l}
        \E{ \int_{ 0 }^{ T } \left\|L_t\right\|_{\HS}^2dt  }< +\infty;
      \end{equation}

    \item [(iii)] there exists a continuous square-integrable $(\F_t^{X,M})$-martingale $M_t$, $t \in[0,T]$, in $\L_2$ such that for every $\varphi \in \Cf_{\alpha_0}^2[0,1]$ 
      \begin{equation} 
  \label{equ_martingale_problem_for_limit_of_mn}
              \begin{split}
	        \langle X_t , \varphi \rangle&= \langle g , \varphi \rangle+ \frac{1}{ 2 }\int_{ 0 }^{ t } \left\langle X_s , \varphi'' \right\rangle ds+\int_{ 0 }^{ t }\langle a_s , \varphi \rangle ds \\
		&+\int_{ 0 }^{ t } \langle f(X_s) , \varphi \rangle ds   +\langle M_t , \varphi \rangle , \quad t \in[0,T],
	      \end{split}
      \end{equation}
      and 
      \[
	\left[ \langle M_{\cdot } , \varphi \rangle \right]_t=\int_{ 0 }^{ t } \|L_s \varphi\|^2ds, \quad t \in[0,T].
      \]
  \end{enumerate}
\end{proposition}

\begin{remark} 
  \label{rem_measurability_of_a}
  Due to equality~\eqref{equ_martingale_problem_for_limit_of_mn} and Theorem~1.2~\cite{Vakhania:1987}, the process $\int_{ 0 }^{ t } a_sds$, $t \in[0,T]$, is $(\F^{X,M})$-measurable.
\end{remark}

\begin{proof} %
  We first remark that the families $\left\{ \lambda^n \I_{\{X^n=0\}},\ n\geq 1 \right\}$, $\left\{ \I_{\{X^n>0\}},\ n\geq 1 \right\}$ and $\left\{ \Gamma^n,\ n\geq 1 \right\}$ are tight due to the compactness of the spaces where they are defined. Consequently, by Proposition~\ref{pro_tightness_of_tilde_x} and Proposition~3.2.4~\cite{Ethier:1986}, the family $\left\{ \left( \tilde{X}^n,\lambda^n \I_{\{X^n=0\}}, \I_{\{X^n>0\}},\Gamma^n \right),\ n\geq 1 \right\}$ is also tight. By the Prokhorov theorem, there exists a subsequence $N \subset \N$ such that
  \[
    \left( \tilde{X}^n,\lambda^n \I_{\{X^n=0\}}, \I_{\{X^n>0\}},\Gamma^n \right) \to (\tilde{X},a,\sigma,\Gamma)
  \]
  in distribution along $N$. Without loss of generality, let us assume that $N=\N$.

  Since 
  \begin{align*}
    \max\limits_{ t \in[0,T] }\big\|\tilde{X}^n_t-X^n_t\big\|^2&= \max\limits_{ t \in[0,T] }\max\limits_{ k \in [n] }\big|x_k^n(t)-x_{k-1}^n(t)\big|^2\\
    &\leq \max\limits_{ t \in[0,T] }\sup\limits_{ 0\leq \delta\leq \frac{1}{ n } }\max\limits_{ |u-u'|\leq \delta }\big|\tilde{X}^n_t(u)-\tilde{X}^n_t(u')\big|^2,
  \end{align*}
  it is easily to see that $\max\limits_{ t \in[0,T] }\big\|\tilde{X}^n_t-X_t^n\big\|^2 \stackrel{d}{\to} 0$, by Skorokhod representation Theorem~3.1.8~\cite{Ethier:1986} and the uniform convergence of $\tilde{X}^n$ to $\tilde{X}$. Hence, $\max\limits_{ t \in[0,T] }\big\|\tilde{X}^n_t-X_t^n\big\|^2 \to  0$ in probability as $n\to\infty$. Using Corollary~3.3.3~\cite{Ethier:1986} and the fact that $\tilde{X}^n$, $n\geq 1$, also converges to $\tilde{X}$ in $\Cf\left( [0,T], \L_2 \right)$ in distribution, we have that $X^n \stackrel{d}{\to} \tilde{X}=:X$ in $\Cf\left( [0,T],\L_2 \right)$.

  We next note that 
  \begin{equation} 
  \label{equ_equality_for_sequence_of_lambda}
    \lambda^n \I_{\{X^n_t=0\}}=\left( \lambda^n-\lambda \right)\I_{\{X^n_t=0\}}+\lambda\left( 1-\I_{\{X^n_t>0\}} \right), \quad t \in[0,T].
  \end{equation}
  By Lemma~\ref{lem_convergence_of_lambdan}, $\left( \lambda^n-\lambda \right)\I_{\{X^n=0\}} \to 0$ in $\rB\left(\L_2\right)$ a.s. as $n\to\infty$. Thus,~\eqref{equ_equality_for_sequence_of_lambda} yields $\lambda^n \I_{\{X^n=0\}} \stackrel{d}{\to}\lambda\left( 1-\sigma \right)$ in $\rB\left(\L_2\right)$, $n\to\infty$. This implies the equality $a=\lambda\left( 1-\sigma \right)$ a.s.

  The existence of a convergent subsequence of $\left\{\bX^n\right\}_{n\geq 1}$,  and statement (i) are proved.

  The statement (ii) directly follows from Lemma~\ref{lem_square_root_operator}.

  In order to prove statement (iii) of the proposition, we first define the following  $\L_2$-valued martingale
\begin{align*}
  M^n_t:&= \sum_{ k=1 }^{ n } \int_{ 0 }^{ t } \sqrt{ n }\I_{\left\{ x_k^n(s)>0 \right\}}dw_k^n(s) \pi_k^n\\
  &= \int_{ 0 }^{ t } \pr^n \mo{\I_{\left\{ X^n_s>0 \right\}}}QdW_s= \int_{ 0 }^{ t } \left( L_s^n \right)^*dW_s , \quad t \in[0,T].
\end{align*}
Set for $\varphi \in \L_2$
\begin{equation} 
  \label{equ_delta_tilde}
  \tilde{\Delta}^n \varphi:=n^3\sum_{ k=1 }^{ n } \Delta^n \varphi_k^n \pi_k^n,
\end{equation}
where $\varphi_k^n=\langle \varphi , \pi_k^n \rangle$, $\varphi_0^n=\alpha_0 \varphi_1^n$ and $\varphi_{n+1}^n=\alpha_0\varphi_n^n$. Since $X^n=\sum_{ k=1 }^{ n } x_k^n\pi_k^n$ and the family $\{ x_k^n,\ k \in [n] \}$ solves SDE~\eqref{equ_system_of_sde}-\eqref{equ_boundary_conditions_for_sde}, we get that for every $\varphi \in \L_2$
\begin{equation} 
  \label{equ_equality_for_mn}
  \begin{split}
    \left\langle M^n_t , \varphi \right\rangle&=\left\langle M^n_t , \pr^n \varphi \right\rangle\\
    &= \langle X^n_t , \pr^n\varphi \rangle-\langle g^n , \pr^n\varphi \rangle- \frac{1}{ 2 }\int_{ 0 }^{ t } \left\langle \tilde{\Delta}^nX_s^n , \pr^n\varphi \right\rangle ds\\
    &- \int_{ 0 }^{ t } \left\langle \lambda^n \I_{\{X^n_s=0\}} , \pr^n\varphi \right\rangle ds -\int_{ 0 }^{ t }\left\langle f(X_s^n) , \pr^n\varphi \right\rangle ds\\
  &=  \langle X^n_t , \varphi \rangle-\langle g^n , \varphi \rangle- \frac{1}{ 2 }\int_{ 0 }^{ t } \left\langle X_s^n , \tilde{\Delta}^n\varphi \right\rangle ds\\
  &- \int_{ 0 }^{ t } \left\langle \lambda^n \I_{\{X^n_s=0\}} , \varphi \right\rangle ds -\int_{ 0 }^{ t }\left\langle f(X_s^n) , \varphi \right\rangle ds, \quad t \in[0,T],
  \end{split}
\end{equation}
and the quadratic variation of the $(\F^{X^n})$-martingale $\left\langle M_t^n , \varphi \right\rangle$ equals
\[
  \left[ \left\langle M^n_{\cdot } , \varphi \right\rangle \right]_t=\int_{ 0 }^{ t } \left\|Q \I_{\{X^n_s>0\}}\pr^n\varphi\right\|^2ds, \quad t \in[0,T].
\]

Let $\tilde{e}_1(u)=1$, $u \in[0,1]$, and $\tilde{e}_k(u)=\sqrt{ 2 }\cos \pi(k-1)u$, $u \in[0,1]$, $k\geq 2$, if $\alpha_0=1$; and $\tilde{e}_k(u)=\sqrt{ 2 }\sin \pi ku$, $u \in[0,1]$, $k\geq 1$, if $\alpha_0=0$. Then $\tilde{e}_k \in \Cf^2_{\alpha_0}[0,1]$ for all $k \geq 1$, and $\{ \tilde{e}_k,\ k\geq 1 \}$ form an orthonormal basis in $\L_2$. Since $\left\|Q \I_{\{X^n_t>0\}}\pr^n \tilde{e}_k\right\|^2\leq \|Q\|^2$, $t \in [0,T]$, $k\geq 1$, the families $\{\left\langle M_{\cdot }^n , \tilde{e}_k \right\rangle, \ n\geq 1\}$ and $\left\{ \left[ \left\langle M^n_{\cdot } , \tilde{e}_k \right\rangle \right], n\geq 1 \right\}$ are tight in $\Cf[0,T]$ for every $k\geq 1$, by the Aldous tightness criterion. According to the tightness of $\left\{ \bX^n,\ n\geq 1 \right\}$, we also have that $\left\{ \left\langle X^n_{\cdot } , \tilde{e}_k \right\rangle,\ n\geq 1 \right\}$ is tight in $\Cf[0,T]$ for each $k\geq 1$. Using Proposition~2.4~\cite{Ethier:1986} and Prokhorov's theorem, we can choose a subsequence $N \subset \N$ such that
\begin{equation} 
  \label{equ_convergence_barm_barv}
  \left(\left\langle X^n_{\cdot } , \tilde{e}_k \right\rangle,\left\langle M_{\cdot }^n , \tilde{e}_k \right\rangle,\left[ \left\langle M_{\cdot }^n , \tilde{e}_k \right\rangle \right]\right)_{k\geq 1} \to \left(\bar{X}_k,\bar{M}_k, \bar{V}_k\right)_{k\geq 1}
\end{equation}
in $\left((\Cf[0,T])^3\right)^{\N}$ in distribution along $N$. In particular, we have that $\left\langle M^n_{\cdot } , \tilde{e}_k \right\rangle^2-\left[ \left\langle M^n_{\cdot } , \tilde{e}_k \right\rangle \right]$, $n \geq 1$, is a sequence of martingales which converges to $\bar{M}_k^2-\bar{V}_k$ in $\Cf[0,T]$ in distribution along $N$ for all $k\geq 1$. 

We fix $m\geq 1$ and let $(\bar{\F}^{\bar{X},\bar{M}, \bar{V},m})$ be the complete right continuous filtration generated by $(\bar{X}_k,\bar{M}_k,\bar{V}_k)$, $k \in [m]$. By Proposition~IX.1.17~\cite{Jacod:2003}, we can conclude that $\bar{M}_k$ and $\bar{M}_k^2-\bar{V_k}$ are continuous local $(\bar{\F}^{\bar{X},\bar{M},\bar{V},m})$-martingales for all $k \in [m]$. Since
\[
  \E{ \left\langle M^n_T , e_k \right\rangle^2 }=\int_{ 0 }^{ T } \E{ \left\|Q \I_{\{X^n_s>0\}}\tilde{e}_k\right\|^2 }ds\leq \|Q\|^2T,
\]
we have that $\E{ \bar{M}_k^2(T) }< +\infty$, by Lemma~4.11~\cite{Kallenberg:2002}. Hence $\bar{M}_k^2$ is a continuous square-integrable $(\bar{\F}^{\bar{X},\bar{M},\bar{V},m})$-martingale with quadratic variation $\left[ \bar{M}_k \right]=\bar{V}$, $k \in [m]$. From Theorem~17.17~\cite{Kallenberg:2002}, we can conclude that $\F^{\bar{X},\bar{M},\bar{V},m}_t=\bar{\F}^{\bar{X},\bar{M},m}_t$, $t \in[0,T]$, where $(\bar{\F}^{\bar{X},\bar{M},m}_t)_{t \in[0,T]}$ is the complete right continuous filtration generated by $\left(\bar{X}_k,\bar{M}_k\right)$, $k \in [m]$. Since for every $t \in[0,T]$ the $\sigma$-algebra $\bar{\F}_t^{\bar{X},\bar{M},m}$ increases to $\bar{\F}_t^{\bar{X},\bar{M}}$ as $m\to\infty$, Theorem~1.6~\cite{Liptser:2001} yields that $\bar{M}_k$ is a continuous square-integrable $(\bar{\F}^{\bar{X},\bar{M}}_t)$-martingale with quadratic variation $[\bar{M}_k]=\bar{V}_k$ for each $k\geq 1$.

Next, we recall that 
\[
  \left( \tilde{X}^n,X^n,\lambda^n \I_{\{X^n=0\}},\Gamma^n \right) \to \left( X,X,a,\Gamma \right)
\]
in $\Cf\left( [0,T],\Cf[0,1] \right)\times \Cf\left( [0,T],\L_2 \right)\times \rB\left(\L_2\right)\times \rB\left(\HS\right)$ in distribution as $n\to\infty$. By Skorokhod representation Theorem~3.1.8~\cite{Ethier:1986}, we may assume that this sequence converges almost surely. Therefore, for every $t \in[0,T]$ and $k\geq 1$
\begin{align*}
  \left\langle X_t^n , \tilde{e}_k \right\rangle \to \left\langle X_t , \tilde{e}_k \right\rangle=:X_k(t), &\quad \left\langle g^n , \tilde{e}_k \right\rangle \to \left\langle g , \tilde{e}_k \right\rangle,\\
  \int_{ 0 }^{ t } \left\langle \lambda^n \I_{\{X^n_s=0\}} , \tilde{e}_k \right\rangle ds &\to \int_{ 0 }^{ t } \left\langle a_s , \tilde{e}_k \right\rangle ds,\\
  \int_{ 0 }^{ t } \left\langle f(X_s^n) , \tilde{e}_k \right\rangle ds & \to \int_{ 0 }^{ t } \left\langle f(X_s) , \tilde{e}_k \right\rangle ds,\\  
  \left[ \left\langle M^n_{\cdot } , \tilde{e}_k \right\rangle \right]_t=\int_{ 0 }^{ t }  \left\|L^n_s \tilde{e}_k\right\|^2 ds &\to \int_{ 0 }^{ t } \left\|L_s \tilde{e}_k\right\|^2ds =:V_k(t)
\end{align*}
a.s. as $n\to\infty$. Using Taylor's formula and the fact that $\tilde{e}_k \in \Cf^3_{\alpha_0}[0,1]$, $k \geq 1$, it is easy to see that for every $t \in[0,T]$ and $k\geq 1$
\[
  \int_{ 0 }^{ t } \left\langle X_s^n , \tilde{\Delta}^n \tilde{e}_k \right\rangle ds \to \int_{ 0 }^{ t } \left\langle X_s , \tilde{e}_k'' \right\rangle ds \quad \mbox{a.s. \ as}\ \ n\to\infty.  
\]
Consequently, for every $t \in[0,T]$ the sequence $\left\langle M^n_t , \tilde{e}_k \right\rangle$, $n\geq 1$, converges to 
\begin{align*}
  M_k(t):&= \left\langle X_t , \tilde{e}_k \right\rangle-\left\langle g , \tilde{e}_k \right\rangle- \frac{1}{ 2 }\int_{ 0 }^{ t } \left\langle X_s , \tilde{e}_k'' \right\rangle ds\\
  &=\int_{ 0 }^{ t } \langle a_s , \tilde{e}_k \rangle ds-\int_{ 0 }^{ t } \left\langle f(X_s) , \tilde{e}_k \right\rangle ds  
\end{align*}
a.s. as $n\to\infty$. Thus, for every $m\in \N$ and $t_i \in [0,T]$, $i \in [m]$,
\[
  \left( \left(\left\langle X^n_{t_i} , \tilde{e}_k \right\rangle,\left\langle M^n_{t_i} , \tilde{e}_k \right\rangle,\left[ \left\langle M^n_{\cdot } , \tilde{e}_k \right\rangle \right]_{t_i}\right)_{i \in [m]} \right)_{k\geq 1} \to \left( \left(X_k(t_i),M_k(t_i),V_k(t_i)\right)_{i \in [m]} \right)_{k\geq 1}
\]
in $\left(\R^{3m}\right)^{\N}$ a.s. as $n\to\infty$. This and convergence~\eqref{equ_convergence_barm_barv} imply that 
\[
  \law\left\{ \left(X_k, M_k,V_k \right)_{k\geq 1} \right\}=\law\left\{ \left(\bar{X}_k,\bar{M}_k,\bar{V}_k\right)_{k\geq 1} \right\}
\]
in $\left( \left( \Cf[0,1] \right)^3 \right)^{\N}$. Consequently, for every $k\geq 1$ the process $M_k$ is a continuous square-integrable $(\bar{\F}^{X,M})$-martingale with quadratic variation 
\[
  \left[ M_k \right]_t=V_k(t)=\int_{ 0 }^{ t } \left\|L_s \tilde{e}_k\right\|^2ds, \quad t \in[0,T],
\]
where $(\bar{\F}_t^{X,M})_{t \in[0,T]}$ is the complete right continuous filtration generated by $X_k,M_k$, $k\geq 1$.

Now we introduce the following process in $\L_2$
\begin{equation} 
  \label{equ_definition_of_m_in_l2}
  M_t:=\sum_{ k=1 }^{ \infty } M_k(t)\tilde{e}_k, \quad t \in[0,T].
\end{equation}
Remark that $M_t$, $t \in[0,T]$, is a well-defined continuous process in $\L_2$. Indeed, by the Burkholder-Davis-Gundy inequality, Lemma~\ref{lem_convergence_of_composition_in_hilbert_shmidt_space},~\eqref{equ_boundedness_of_expectation_of_norm_of_l} and the dominated convergence theorem, for every $n,m\geq 1$
\begin{align*}
  \e\Bigg[\max\limits_{ t \in[0,T] }&\bigg\|\sum_{ k=1 }^{ n } M_k(t)\tilde{e}_k-\sum_{ k=1 }^{ n+m } M_k(t)\tilde{e}_k\bigg\|^2 \Bigg]=\E{ \max\limits_{ t \in[0,T] }\sum_{ k=n+1 }^{ n+m } M_k^2(t) }\\
  &\leq \int_{ 0 }^{ T } \E{ \sum_{ k,l=n+1 }^{n+m}\left\langle L_t \tilde{e}_k , L_t\tilde{e}_l  \right\rangle}dt\\
  &=  \int_{ 0 }^{ T }  \E{\sum_{ k,l=1 }^{ \infty } \left\langle L_t\tpr^{n,n+m}\tilde{e}_k , L_t\tpr^{n,n+m}\tilde{e}_l \right\rangle} dt\\
  &=  \int_{ 0 }^{ T } \E{ \left\|L_t\tpr^{n,n+m}\right\|_{\HS}^2 }dt \to 0
\end{align*}
as $n,m\to\infty$, where $\tpr^{n,n+m}$ is the orthogonal projection in $\L_2$ onto $\spann\{ \tilde{e}_k,\ k =n+1,\dots,n+m \}$.
This implies the convergence of series~\eqref{equ_definition_of_m_in_l2} and the continuity of $M_t$, $t \in[0,T]$, in $\L_2$.

Since $\bar{\F}^{X,M}_t=\F^{X,M}_t$, $t \in[0,T]$, and $\left\langle M_t , \tilde{e}_k \right\rangle=M_k(t)$, $t \in[0,T]$, for all $k\geq 1$, it is easily seen that $M$ is a continuous square-integrable $(\F_t^{X,M})$-martingale in $\L_2$ with quadratic variation 
\[
  \left[ M \right]_t=\int_{ 0 }^{ t } L_s^2 ds=\int_{ 0 }^{ t } \Gamma_s ds , \quad t \in[0,T]. 
\]
This implies statement (iii). The proposition is proved.
\end{proof}

\subsection{A property of quadratic variation of heat semi martingales}%
\label{sub:a_property_of_quadratic_variation_of_some_semi_martingales}

In this section, we will assume that $(\Omega,\F,(\F_t)_{t\geq 0},\p)$ is a filtered complete probability space, where the filtration $(\F_t)_{t\geq 0}$ is complete and right continuous. Let $T>0$ be fixed. Consider a continuous $(\F_t)$-adapted $\L_2$-valued process $Z_t$, $t \in[0,T]$, such that there exist random elements $a$ and $L$ in $\LL{ \L_2}$ and $\LL{ \HS}$, respectively, such that for every $\varphi \in \Cf^2_{\nd}[0,1]$ the processes $\int_{ 0 }^{ t } \langle a_s , \varphi \rangle ds $, $t \in[0,T]$, and $\int_{ 0 }^{ t }  \|L_s \varphi\|^2 ds  $, $t \in[0,T]$, are $(\F_t)$-adapted, and
\begin{equation} 
  \label{equ_heat_semimartingale}
  \M^{\varphi}_Z(t):=\langle Z_t , \varphi \rangle-\langle Z_0 , \varphi \rangle- \frac{1}{ 2 }\int_{ 0 }^{ t } \langle Z_s , \varphi'' \rangle ds+\int_{ 0 }^{ t } \langle a_s , \varphi \rangle ds, \quad t \in[0,T],
\end{equation}
is a local $(\F_t)$-martingale with quadratic variation 
\[
  [\M^{\varphi}_Z]_t=\int_{ 0 }^{ t } \|L_s \varphi\|^2 ds, \quad t \in[0,T].
\]
Note that the assumptions on $L$ implies that the continuous process $\int_{ 0 }^{ t } \left\|L_s\right\|_{\HS}^2ds $, $t \in[0,T]$, is well-defined and $(\F_t)$-adapted.

We will further consider the case of the Neumann boundary condition, where $\alpha_0=1$. All conclusions of this section will be the same for the Dirichlet boundary condition, where $\alpha_0=0$. Let $\left\{ \tilde{e}_k,\ k\geq 1 \right\}$ be the family of the eigenfunctions of $\Delta$ on $[0,1]$ with Neumann boundary conditions. We recall that $\tilde{e}_1(u)=1$, $u \in[0,1]$, and $\tilde{e}_k(u)=\sqrt{ 2 }\cos \pi (k-1) u$, $u \in[0,1]$, $k \geq 2$.  Denote the orthogonal projection in $\L_2$ onto $\spann\{ \tilde{e}_k,\ k \in [n] \}$ by $\tpr^n$.

Let $Z_t^n=\tpr^nZ_t$, $t\geq 0$, and $a_t^n=\tpr^na_t$, $t \in[0,T]$. We also introduce 
\[
  \dot{Z}^n_t=\sum_{ k=1 }^{ n } \langle Z_t , \tilde{e}_k \rangle \tilde{e}'_k, \quad t \in[0,T],\ \ n\geq 1,
\]
and note that $\dot{Z}^n$, $n\geq 1$, is a sequence of random elements in $\LL{ \L_2}$.

\begin{lemma} 
  \label{lem_derivative_of_z}
  \begin{enumerate}
    \item [(i)] The equality
      \[
	\P{ \dot{Z}^n,\ n\geq 1,\ \mbox{converges in } \LL{ \L_2}\ \mbox{and a.e. as}\ n\to\infty }=1
      \]
      holds.

    \item [(ii)] Set 
      \[
	\dot{Z}:=\lim_{ n\to\infty }\dot{Z}^n,
      \]
      where the limit is taken a.e. Then $\dot{Z}$ is a random element in $\LL{ \L_2}$ and for every $t \in[0,T]$
      \[
	\int_{ 0 }^{ t } \|\dot{Z}^n_s\|^2ds \to \int_{ 0 }^{ t } \|\dot{Z}_s\|^2ds \quad \mbox{a.s. \ as}\ \ n\to\infty.  
      \]

  \end{enumerate}
\end{lemma}

\begin{proof} %
  Set $z_k(t):=\langle Z_t , \tilde{e}_k \rangle$, $t \in[0,T]$, $k\geq 1$. Then, by the definition of $Z$, for every $k\geq 1$ the process 
  \[
    \xi_k(t):=z_k(t)-z_k(0)+ \frac{\pi^2 (k-1)^2}{ 2 }\int_{ 0 }^{ t } z_k(s)ds+\int_{ 0 }^{ t } a_k(s)ds, \quad t \in[0,T],
  \]
  is a continuous local $(\F_t)$-martingale with quadratic variation 
  \[
    [\xi_k]_t=\int_{ 0 }^{ t } \|L_s \tilde{e}_k\|^2ds, \quad t \in[0,T], 
  \]
  where $a_k(s):=\langle a_s , \tilde{e}_k \rangle$. Denote $\sigma_{k,l}^2(t):= \langle L_t \tilde{e}_k , L_t \tilde{e}_l \rangle$, $t \in[0,T]$, and note that
  \[
    Z_t^n=\sum_{ k=1 }^{ n } z_k(t)\tilde{e}_k \quad \mbox{and} \quad a_t^n=\sum_{ k=1 }^{ n } a_k(t)\tilde{e}_k,\quad t \in[0,T], \ \ n\geq 1.
  \]
  By the \Ito formula and the polarisation equality, we get
  \begin{equation} 
  \label{equ_ito_formula_for_zn_square}
    \begin{split}
      \|Z_t^n\|^2&=  \|Z_0^n\|^2- \int_{ 0 }^{ t } \|\dot{Z}_s^n\|^2ds + 2\int_{ 0 }^{ t } \left\langle a^n_s , Z_s^n \right\rangle ds\\
      &+ \int_{ 0 }^{ t } \|L_s\tpr^n\|^2_{\HS}ds + \M^n(t), \quad t \in[0,T],
    \end{split}
  \end{equation}
  where $\M^n(t)$, $t \in[0,T]$, is a continuous local $(\F_t)$-martingale defined as 
  \[
    \M^n(t)=2\sum_{ k=1 }^{ n } \int_{ 0 }^{ t }z_k(s) d\xi_k(s), \quad t \in[0,T].  
  \]
  A simple computation gives that
  \[
    [\M^n]_t=4 \int_{ 0 }^{ t } \|L_s Z^n_s\|^2 ds, \quad t \in[0,T]. 
  \]
  Trivially, $\|Z_t^n\|^2 \to \|Z_t\|^2$ a.s. as $n\to\infty$ for all $t \in[0,T]$. Using the dominated convergence theorem, we can conclude that $\int_{ 0 }^{ t } \left\langle a^n_s , Z^n_s \right\rangle ds \to \int_{ 0 }^{ t } \left\langle a_s , Z_s \right\rangle ds $ a.s. as $n\to\infty$. Next, by Lemma~\ref{lem_convergence_of_composition_in_hilbert_shmidt_space} and the dominated convergence theorem, $\int_{ 0 }^{ t } \|L_s\tpr^n\|^2_{\HS}ds \to \int_{ 0 }^{ t } \left\|L_s\right\|_{\HS}^2ds $ a.s. as $n\to\infty$. Next, we will show that $\M^n(t)$ converges in probability. Since $\M^n$ is a local martingale, we need to choose a localization sequence of $(\F_t)$-stopping times defined as follows
  \[
    \tau_k:=\inf\left\{ t \in[0,T]:\ \int_{ 0 }^{ t } \left\|L_s\right\|_{\HS}^2ds \geq k  \right\}\wedge T.
  \]
  Then the processes $\M^n(t \wedge\tau_k)$, $t \in[0,T]$, $n\geq 1$, are square-integrable $(\F_t)$-martingales for every $k\geq 1$, and $\tau_k \uparrow T$ as $k\to\infty$. By the Burkholder-Davis-Gundy inequality (see e.g.~\cite[Theorem~III.3.1]{Ikeda:1989}), for every $k,n,m\geq 1$, $n<m$,
  \[
    \e\Bigg[\max\limits_{ t \in [0,T] }\big|\M^n(t\wedge \tau_k)-\M^{m}(t\wedge\tau_k)\big|^2\Bigg]\leq  16 \E{ \int_{ 0 }^{ \tau_k } \left\|L_s\tpr^{n,m}Z_s^m\right\|^2ds  },
  \]
  where $\tpr^{n,m}$ is the orthogonal projection in $\L_2$ onto $\spann\left\{ \te_k,\ k=n+1,\dots,m \right\}$.  Hence, by the dominated convergence theorem,
  \[
    \E{ \max\limits_{ t \in [0,T] }\left|\M^n(t\wedge \tau_k)-\M^{m}(t\wedge\tau_k)\right|^2 }\to 0 \quad \mbox{as}\ \ n\to\infty.
  \]
  This implies that there exists a continuous square-integrable $(\F_t)$-martingale $\M_k(t)$, $t \in[0,T]$, such that 
  \[
    \max\limits_{ t \in[0,T] }\big|\M^n(t \wedge \tau_k)-\M_k(t)\big| \to 0 \quad \mbox{in probability as}\ \ n\to\infty.
  \]
  By Lemma~B.11~\cite{Cherny:2005},
  \[
    \left[ \M_k \right]_t=4\int_{ 0 }^{ t\wedge\tau_k } \left\|L_sZ_s\right\|^2 ds, \quad t \in[0,T]. 
  \]
  Furthermore, for every $k\geq 1$ $\M_k=\M_{k+1}(\cdot \wedge \tau_k)$ a.s. We define $\M(t):=\M_k(t)$ for $t\leq \tau_k$, $k\geq 1$. Trivially, $\M$ is a continuous local $(\F_t)$-martingale with quadratic variation
  \[
    \left[ \M \right]_t=4\int_{ 0 }^{ t } \left\|L_sZ_s\right\|^2 ds, \quad t \in[0,T]. 
  \]
  Using Lemma~4.2~\cite{Kallenberg:2002}, $\M^n(t) \to \M(t)$ in probability as $n\to\infty$ for every $t \in[0,T]$. 

  We have obtained that every term, except $\int_{ 0 }^{ t } \|\dot{Z}_s^n\|^2 ds$, of equality~\eqref{equ_ito_formula_for_zn_square} converges in probability. Hence, $\int_{ 0 }^{ t } \|\dot{Z}^n_s\|^2 ds$ also converges in probability. Moreover, this sequence is monotone. By Lemma~4.2~\cite{Kallenberg:2002}, it converges almost surely. By Fatou's lemma,
  \begin{equation} 
  \label{equ_l2_boundedness_of_dotz}
  \int_{ 0 }^{ T } \lim_{ n\to\infty }\|\dot{Z}^n_s\|^2 ds <\infty \quad \mbox{a.s.} 
  \end{equation}
  This implies the convergence of $\left\{\dot{Z}^n_s(\omega)\right\}_{n\geq 1}$ in $\L_2$ for almost all $s$ and $\omega$. Hence $\dot{Z}^n$, $n\geq 1$, converges to $\dot{Z}$ a.e. a.s. as $n\to\infty$. The equality in the second part of the lemma follows from the monotone convergence theorem and~\eqref{equ_l2_boundedness_of_dotz}. In particular, $\|\dot{Z}^n\|_{\L_2,T} \to \|\dot{Z}\|_{\L_2,T}$. Thus, $\dot{Z}^n \to \dot{Z}$ in $\LL{ \L_2}$ a.s., according to Proposition~2.12~\cite{Kallenberg:2002}.
\end{proof}

\begin{proposition} %
  \label{pro_ito_formula}
  Let $F \in \Cf^2\left( \R  \right)$ has a bounded second derivative and $h \in \Cf^1[0,T]$. Then
  \begin{equation} 
  \label{equ_ito_formula_for_fzn}
    \begin{split}
      \langle F(Z_t) , h \rangle &=  \langle F(Z_0) , h \rangle- \frac{1}{ 2 }\int_{ 0 }^{ t } \left\langle \left(F'(Z_s)h\right)' , \dot{Z}_s \right\rangle ds +     \int_{ 0 }^{ t } \left\langle F'(Z_s)h , a_s \right\rangle ds \\
      &+ \frac{1}{ 2 }\int_{ 0 }^{ t } \left\langle L_s \mo{F''(Z_s)h} , L_s \right\rangle_{\HS} ds +\M_{F,h}(t), \quad t \in[0,T],
    \end{split}
  \end{equation}
  where $\M_{F,h}(t)$, $t \in [0,T]$, is a continuous local $(\F_t)$-martingale with quadratic variation 
  \[
    \left[ \M_{F,h} \right]_t= \int_{ 0 }^{ t } \left\|L_sF'(Z_s)h\right\|^2 ds, \quad t \in[0,T], 
  \]
  and $\left(F'(Z_s)h\right)':=F''(Z_s)\dot{Z}_sh+F'(Z_s)h' \in \L_2$.
\end{proposition}

\begin{proof} %
  As in the proof of Lemma~\ref{lem_derivative_of_z}, we can compute for every $n\geq 1$ 
  \begin{align*}
    \left\langle F\left( Z_t^n \right) , h \right\rangle&= \langle F(Z^n_0) , h \rangle-\sum_{ k=1 }^{ n } \frac{ \pi^2 (k-1)^2 }{ 2 }\int_{ 0 }^{ t } \left\langle F'(Z_s^n)h , \tilde{e}_k \right\rangle z_k(s)ds\\
    &+ \sum_{ k=1 }^{ n } \int_{ 0 }^{ t } \left\langle F'(Z_s^n)h , \tilde{e}_k \right\rangle a_k(s)ds + \frac{1}{ 2 }\sum_{ k,l=1 }^{ n } \int_{ 0 }^{ t } \left\langle F''(Z_s^n)h \tilde{e}_k, \tilde{e}_l \right\rangle \sigma_{k,l}^2(s)ds\\
    &+ \sum_{ k=1 }^{ n } \int_{ 0 }^{ t } \left\langle F'(Z_s^n)h , \tilde{e}_k \right\rangle d\xi_k(s), \quad t \in[0,T]. 
  \end{align*}
  Consequently,
  \begin{align*}
    \langle F(Z_t^n) , h \rangle &=  \langle F(Z_0^n) , h \rangle- \frac{1}{ 2 }\int_{ 0 }^{ t } \left\langle \left(F'(Z_s^n)h\right)'_n , \dot{Z}_s^n \right\rangle ds +     \int_{ 0 }^{ t } \left\langle F'(Z_s^n)h , a_s^n \right\rangle ds \\
    &+ \int_{ 0 }^{ t } \left\langle L_s\tpr^n\mo{F''(Z_s^n)h} , L_s\tpr^n \right\rangle_{\HS} ds +\M^n_{F,h}(t), \quad t \in[0,T],
  \end{align*}
  where  
  \[
    \M^n_{F,h}(t)= \sum_{ k=1 }^{ n } \int_{ 0 }^{ t } \left\langle F'(Z_s^n)h , \tilde{e}_k \right\rangle d \xi_k(s), 
  \]
  and $\left( F'(Z_s^n)h \right)'_n=\sum_{ k=1 }^{ n }\left\langle F'(Z_s^n)h , \tilde{e}_k \right\rangle \tilde{e}_k' $.
  The process $\M_{F,h}^n(t)$, $t \in[0,T]$, is a continuous local $(\F_t)$-martingale with quadratic variation 
  \begin{align*}
    \left[ \M_{F,h}^n(t) \right]_t&= \sum_{ k,l=1 }^{ n } \int_{ 0 }^{ t } \left\langle F'(Z_s^n)h , \tilde{e}_k \right\rangle \left\langle F'(Z_s^n)h , \tilde{e}_l \right\rangle \sigma_{k,l}^2 ds \\
    &= \int_{ 0 }^{ t } \left\|L_s\tpr^nF'(Z_s^n)h\right\|^2 ds, \quad t \in[0,T]. 
  \end{align*}
  By the boundedness of the second derivative of $F$ we have that there exists a constant $C>0$ such that $|F'(x)|\leq C(1+|x|)$ and $|F(x)|\leq C(1+|x|^2)$. Therefore, $\langle F(Z_t^n) , h \rangle \to \langle F(Z_t) , h \rangle$ a.s., $n\to\infty$, and $F'(Z_t^n)h \to F'(Z_t)h$ and $F''(Z_t^n)h \to F''(Z_t)h$ in $\L_2$ a.s. as $n\to\infty$ for all $t \in[0,T]$. By the dominated convergence theorem and Lemma~\ref{lem_convergence_of_composition_in_hilbert_shmidt_space},
  \[
    \int_{ 0 }^{ t } \left\langle F'(Z_s^n)h , a_s^n \right\rangle ds \to \int_{ 0 }^{ t } \left\langle F'(Z_s)h , a_s \right\rangle \quad \mbox{a.s.}
  \]
  and
  \[
    \int_{ 0 }^{ t } \left\langle L_s\tpr^n\mo{F''(Z_s^n)h},L_s\tpr^n \right\rangle_{\HS}ds \to  \int_{ 0 }^{ t	} \left\langle L_s\mo{F''(Z_s)h}, L_s \right\rangle_{\HS} ds \quad \mbox{a.s.}
  \]
  as $n\to\infty$. Using the localization sequence, one can show that for every $t \in[0,T]$ $\M^n_{F,h}(t) \to \M_{F,h}(t)$ in probability as in the proof of the previous lemma.
  
  In order to finish the proof of the proposition, we only need to show the convergence $\int_{ 0 }^{ t } \left\langle \left( F'(Z_s^n)h \right)'_n , \dot{Z}_s^n \right\rangle ds $ to the corresponding term. By Lemma~\ref{lem_derivative_of_z}, it is enough to show that $\left( F'(Z_{\cdot }^n)h \right)'_n \to \left( F'(Z_{\cdot })h \right)'=F''(Z_{\cdot })\dot{Z}_{\cdot }h+F'(Z_{\cdot })h'$ a.e. a.s. as $n\to\infty$. But this easily follows from the integration by parts formula.
\end{proof}


\begin{theorem} 
  \label{the_properties_of_quadratic_variation_of_l2_valued_processes}
  Let the process $Z_t$, $t \in[0,T]$, and the random element $L \in \LL{ \HS}$ be as above. Then the following equality 
  \[
    L_{\cdot }\mo{\I_{\{Z_{\cdot}\not= 0\}}}=L \quad \mbox{in}\ \ \LL{ \HS}\ \ \mbox{a.s.}
  \]
  holds.
\end{theorem}

\begin{proof} %
  In order to prove the theorem, we will use Proposition~\ref{pro_ito_formula}. We fix a function $\psi \in \Cf\left( \R  \right)$ such that $\supp \psi \subset [-1,1]$, $0\leq \psi(x)\leq 1$, $x \in \R $, and $\psi(0)=1$. Define $\psi_{\eps}(x):=\psi\left( \frac{x}{ \eps } \right)$, \ $x \in \R $, \ and 
  \[
    F_{\eps}(x):=\int_{ -\infty }^{ x } \left(\int_{ -\infty }^{ y } \psi_{\eps}(r)dr\right)dy, \quad x \in \R .  
  \]
  Then $0\leq F_{\eps}'(x)\leq 2\eps$, $x \in \R $, and $F_{\eps}''(x)\to \I_{\left\{ 0 \right\}}(x)$ as $\eps \to 0+$ for all $x \in \R $. 
  
  Let a non-negative function $h \in \Cf^1[0,1]$ be fixed. By Proposition~\ref{pro_ito_formula}, 
  \begin{align*}
    \langle F_{\eps}(Z_t) , h \rangle &=  \langle F_{\eps}(Z_0) , h \rangle- \frac{1}{ 2 }\int_{ 0 }^{ t } \left\langle \left(F_{\eps}'(Z_s)h\right)' , \dot{Z}_s \right\rangle ds +     \int_{ 0 }^{ t } \left\langle F_{\eps}'(Z_s)h , a_s \right\rangle ds \\
    &+ \frac{1}{ 2 }\int_{ 0 }^{ t } \left\langle L_s \mo{F_{\eps}''(Z_s)h} , L_s \right\rangle_{\HS} ds +\M_{F_{\eps},h}(t), \quad t \in[0,T],
  \end{align*}
  and the quadratic variation of the local $(\F_t)$-martingale equals
  \[
    \left[ \M_{F_{\eps},h} \right]_t= \int_{ 0 }^{ t } \left\|L_sF_{\eps}'(Z_s)h\right\|^2 ds, \quad t \in[0,T], 
  \]
  Making $\eps \to 0+$, we can immediately conclude that for every $t \in[0,T]$ 
  \[
    \left|\langle F_{\eps}(Z_t) , h \rangle-\langle F_{\eps}(Z_0) , h \rangle\right|\leq 2\eps \|Z_t-Z_0\|\|h\| \to 0 \quad \mbox{a.s.},
  \]
  and 
  \[
    \left| \int_{ 0 }^{ t } \langle F'_{\eps}(Z_s)h , a_s \rangle ds  \right|\leq 2\eps \|h\| \int_{ 0 }^{ t } \|a_s\|ds \to 0 \quad \mbox{a.s.}
  \]
  Similarly to the proof of Lemma~\ref{lem_derivative_of_z}, using the localization sequence, one can show that $M_{F_{\eps},h}(t) \to 0$ in probability. By the dominated convergence theorem and Lemma~\ref{lem_convergence_of_composition_in_hilbert_shmidt_space},
  \[
    \int_{ 0 }^{ t } \left\langle L_s\mo{F''_{\eps}(Z_s)h} , L_s \right\rangle_{\HS}ds \to \int_{ 0 }^{ t } \left\langle L_s\mo{\I_{\left\{ Z_s=0 \right\}}h} , L_s \right\rangle_{\HS}ds \quad \mbox{a.s.}  
  \]
  Again, by the dominated convergence theorem and Lemma~\ref{lem_truncation}, we have 
  \begin{align*}
    \int_{ 0 }^{ t } \left\langle (F'_{\eps}(Z_s)h)' , \dot{Z}_s \right\rangle ds &=  \int_{ 0 }^{ t } \left\langle F''_{\eps}(Z_s)\dot{Z_s}h , \dot{Z}_s \right\rangle ds \\
    &+ \int_{ 0 }^{ t } \left\langle F'(Z_s)h' , \dot{Z}_s \right\rangle ds \to \int_{ 0 }^{ t } \left\|\I_{\left\{ Z_s=0 \right\}}\dot{Z}_s \sqrt{ h }\right\|^2 ds=0 \quad \mbox{a.s.}  
  \end{align*}

  We have obtained that for every $t \in[0,T]$ 
  \[
    \int_{ 0 }^{ t } \left\langle L_s\mo{\I_{\left\{ Z_s=0 \right\}}h} , L_s \right\rangle_{\HS} ds =0 \quad \mbox{a.s.}
  \]
  Then taking $h=1$ and applying Lemma~\ref{lem_norm_of_composition_of_hilbert_shmidt_operators}, it is easy to see that 
  \[
    \int_{ 0 }^{ T } \left\|L_s\mo{\I_{\left\{ Z_s=0 \right\}}}\right\|^2_{\HS}ds=0.
  \]
  The proof of the theorem is completed.
\end{proof}

\subsection{Proof of the existence theorem}
\label{sub:proof_of_theorem}

In this section, we will consider the random element $\bX^n$ defined in Section~\ref{sub:martingale_problem_for_limit_points_of_the_discrete_approximation}. According to Proposition~\ref{pro_subsequence_of_xn}, there exists a subsequence $N \subset \N$ such that 
\[
  \bX^n=\left( \tilde{X}^n,X^n,\lambda^n \I_{\{X^n=0\}}, \I_{\{X^n>0\}},\Gamma^n \right) \to \left( \tilde{X},X,a,\sigma,\Gamma \right) \quad \mbox{in}\ \  \W_{\L_2}
\]
in distribution along $N$. As before, without loss of generality, we may assume that $N=\N$. By the Skorokhod representation theorem, we can assume that this sequence converges almost surely. Since $\tilde{X}^n \to \tilde{X}$ in $\Cf\left( [0,T], \Cf[0,1] \right)$ a.s., and a.s. for all $t \in[0,T]$ the quality $\tilde{X}_t=X_t$ in $\L_2$ holds, the inequality 
\[
  \max\limits_{ t \in[0,T] }\big\|\tilde{X}^n_t-X^n_t\big\|\leq \max\limits_{ t \in[0,T] }\sup\limits_{ 0\leq \delta\leq \frac{1}{ n } }\max\limits_{ |u-u'|\leq \delta }\big|\tilde{X}^n_t(u)-\tilde{X}^n_t(u')\big|^2
\]
implies that 
\begin{equation} 
  \label{equ_convergence_of_xn_to_x}
  \P{ \forall t \in[0,T],\ \ X^n_t \to X_t\ \ \mbox{a.e.} }=1.
\end{equation}

{\bf I.} We will first show that $\Gamma=\mo{\I_{\{X_{\cdot}>0\}}}Q^2 \mo{\I_{\{X_{\cdot}>0\}}}$ a.s. 

Using Proposition~\ref{pro_subsequence_of_xn}~(ii), there exists a random element $L$ in $\LL{ \HS}$ such that $\Gamma=L^2$ a.s. Next, by Proposition~\ref{pro_subsequence_of_xn} and Theorem~\ref{the_properties_of_quadratic_variation_of_l2_valued_processes}, $L_{\cdot } \I_{\{X_{\cdot}>0\}}=L$ a.s. Therefore, using the convergence of $\Gamma^n=\pr^n\I_{\{X^n_{\cdot}>0\}}Q^2 \I_{\{X^n_{\cdot}>0\}}\pr^n$ to $\Gamma$ in $\rB\left(\HS\right)$ a.s., we obtain that for evert $t \in[0,T]\cap \Q$ and $k,l\geq 1$ almost surely
\begin{align*}
  \int_{ 0 }^{ t } \langle \Gamma_s , &e_k\odot e_l \rangle_{\HS} ds= \int_{ 0 }^{ t } \left\langle \Gamma_s e_l , e_k \right\rangle ds= \int_{ 0 }^{ t } \left\langle L_se_l , L_se_k \right\rangle ds \\
  &= \int_{ 0 }^{ t } \left\langle L_s \I_{\{X_s>0\}}e_l , L_s \I_{\{X_s>0\}}e_k \right\rangle ds\\
  &= \int_{ 0 }^{ t } \left\langle \Gamma_s \I_{\{X_s>0\}}e_l , \I_{\{X_s>0\}}e_k \right\rangle ds \\
  &= \lim_{ n\to\infty }\int_{ 0 }^{ t } \left\langle \Gamma_s^n \I_{\{X_s>0\}}e_l , \I_{\{X_s>0\}}e_k \right\rangle ds\\
  &= \lim_{ n\to\infty }\int_{ 0 }^{ t } \left\langle Q \I_{\{X^n_s>0\}}\pr^n\left(\I_{\{X_s>0\}}e_l\right) , Q \I_{\{X^n_s>0\}}\pr^n\left(\I_{\{X_s>0\}}e_k\right) \right\rangle ds\\
  &= \lim_{ n\to\infty }\int_{ 0 }^{ t } \left\langle Q \pr^n\left(\I_{\{X^n_s>0\}}\I_{\{X_s>0\}}e_l\right) , Q \pr^n\left(\I_{\{X^n_s>0\}}\I_{\{X_s>0\}}e_k\right) \right\rangle ds\\
  &= \int_{ 0 }^{ t } \left\langle Q \I_{\{X_s>0\}}e_l , Q \I_{\{X_s>0\}}e_k \right\rangle ds\\
  &= \int_{ 0 }^{ t } \left\langle \I_{\{X_s>0\}}Q^2 \I_{\{X_s>0\}}, e_k\odot e_l \right\rangle_{\HS} ds.
\end{align*}
In the last equality, we have used the fact that $\I_{(0,+\infty)}(x_n)\I_{(0,+\infty)}(x) \to \I_{(0,+\infty)}(x)$ as $x_n \to x$ in $\R $, convergence~\eqref{equ_convergence_of_xn_to_x} and the dominated convergence theorem. Since the family $\left\{ \I_{[0,t]}e_k\odot e_l,\ t \in[0,T]\cap \Q,\ k,l\geq 1 \right\}$ is countable and its linear span is dense in $\LL{ \HS}$, we trivially get that 
\begin{equation} 
  \label{equ_equality_for_gamma}
  \Gamma=\I_{\{X_{\cdot}>0\}}Q^2 \I_{\{X_{\cdot}>0\}}\quad \mbox{a.s.} 
\end{equation}

{\bf II.} Let $\chi^2$ be defined by~\eqref{equ_function_chi}. We next want to show that 
\begin{equation} 
  \label{equ_equality_for_sigma_and_i}
  \I_{\left\{ \chi>0 \right\}}\sigma=\I_{\left\{ \chi>0 \right\}}\I_{\{X>0\}}\quad  \mbox{in}\ \  \LL{ \L_2}\ \ \mbox{a.s.}
\end{equation}
But this will directly follow from the following lemma.

\begin{lemma} 
  \label{lem_convergence_of_sigma}
  Let $Z^n_t$, $t \in [0,T]$, $n\geq 1$, be a sequence of $\L_2$-valued measurable functions such that $Z^n_t\geq 0$ for all $t \in[0,T]$ and $n\geq 1$, and 
  \[
    \leb_T\otimes \leb_1\left\{ (t,u) \in [0,T]\times [0,1]:\ Z^n_t(u) \not\to Z_t(u) \right\}=0.
  \]
  If 
  \begin{equation} 
  \label{equ_convergence_of_qzpr}
    \pr^n\mo{\I_{\{Z^n_{\cdot}>0\}}}Q^2\mo{\I_{\{Z^n_{\cdot}>0\}}}\pr^n \to \mo{\I_{\{Z_{\cdot}>0\}}}Q^2\mo{\I_{\{Z_{\cdot}>0\}}}
  \end{equation}
  in $\rB\left(\HS\right)$, and \ $\I_{\{Z^n>0\}} \to \sigma$ \ in $\rB\left(\L_2\right)$ as $n\to\infty$, then 
  \begin{equation} 
  \label{equ_equality_for_zsig}
    \I_{\left\{ \chi>0 \right\}}\sigma=\I_{\left\{ \chi>0 \right\}}\I_{\{Z>0\}}.
  \end{equation}
\end{lemma}

We postpone the proof of the lemma to the end of this section.
\vspace{3mm}

{\bf III.} Using equality~\eqref{equ_equality_for_sigma_and_i}, Proposition~\ref{pro_subsequence_of_xn}~(i) and assumption~\eqref{equ_the_main_condition_for_existance} of Theorem~\ref{the_existence_of_solution_to_spde}, we get
\begin{equation} 
  \label{equ_equality_for_a}
  \begin{split}
  a&= \lambda(1-\sigma)=\lambda \I_{\left\{ \chi>0 \right\}}\left(1-\sigma\right)\\
  &= \lambda \I_{\left\{ \chi>0 \right\}}\left(1-\I_{\{X>0\}}\right) = \lambda \I_{\left\{ \chi>0 \right\}} \I_{\{X=0\}}=\lambda \I_{\{X=0\}}.
\end{split}
\end{equation}

Hence, using Proposition~\ref{pro_subsequence_of_xn}~(iii) and equalities~\eqref{equ_equality_for_gamma},~\eqref{equ_equality_for_a}, we have that for every $\varphi \in \Cf^2_{\alpha_0}[0,1]$ almost surely
\begin{equation} 
  \label{equ_martingale_problem_for_concrete_a_gamma}
  \begin{split}
    \left\langle X_t , \varphi \right\rangle&= \left\langle g , \varphi \right\rangle + \frac{1}{ 2 }\int_{ 0 }^{ t } \left\langle X_s , \varphi'' \right\rangle ds+ \int_{ 0 }^{ t } \left\langle \lambda \I_{\{X_s=0\}} , \varphi \right\rangle ds\\
    &+ \int_{ 0 }^{ t } \left\langle f(X_s) , \varphi \right\rangle ds + \left\langle M_t , \varphi \right\rangle, \quad t \in[0,T],
  \end{split}
\end{equation}
and $\left\langle M_t , \varphi \right\rangle$, $t \in[0,T]$, is a continuous square-integrable $(\F^{X,M})$-martingale with quadratic variation 
\[
  \left[ \left\langle M_{\cdot} , \varphi \right\rangle \right]_t=\int_{ 0 }^{ t } \left\|Q \I_{\{X_s>0\}}\varphi\right\|^2ds, \quad t \in[0,T].
\]
In particular,~\eqref{equ_martingale_problem_for_concrete_a_gamma} yields that $\F^{X,M}_t=\F^X$, $t \in[0,T]$.

This finishes the proof of Theorem~\ref{the_existence_of_solution_to_spde}.

\begin{proof}[Proof of Lemma~\ref{lem_convergence_of_sigma}] %
  It is easily seen that convergence~\eqref{equ_convergence_of_qzpr} is equivalent to the convergence 
  \[
    \mo{\I_{\{Z^n_{\cdot}>0\}}}Q^2\mo{\I_{\{Z^n_{\cdot}>0\}}} \to \mo{\I_{\{Z_{\cdot }>0\}}}Q^2\mo{\I_{\{Z_{\cdot}>0\}}} \quad \mbox{in}\ \ \rB\left(\HS\right)
  \]
  as $n\to\infty$. So, for every $\varphi \in \LL{ \L_2}$, we have 
   \begin{align*}
     \int_{ 0 }^{ T } \left\|Q \I_{\{Z^n_t>0\}}\varphi_t\right\|^2dt&= \left\langle \mo{\I_{\{Z^n_{\cdot }>0\}}}Q^2 \mo{\I_{\{Z^n_{\cdot }>0\}}}, \varphi_{\cdot }\odot \varphi_{\cdot } \right\rangle_{\HS,T} \\
     & \to \left\langle \mo{\I_{\{Z_{\cdot}>0\}}}Q^2 \mo{\I_{\{Z_{\cdot}>0\}}} , \varphi_{\cdot }\odot \varphi_{\cdot } \right\rangle_{\HS,T}\\
     &=\int_{ 0 }^{ T } \left\|Q \I_{\{Z_t>0\}} \varphi_t\right\|^2dt 
   \end{align*}
   as $n\to\infty$, where $\varphi_{\cdot }\odot \varphi_{\cdot }$ is defined as $\varphi_t\odot \varphi_t$, $t \in[0,T]$. Replacing $\varphi$ by $e_k \I_{\{Z=0\}}$ for every $k\geq 1$, we obtain that
   \begin{equation} 
  \label{equ_convergence_of_qie}
     \int_{ 0 }^{ T } \big\|Q \I_{\{Z^n_t>0\}}\I_{\{Z_t=0\}}e_k\big\|^2 dt \to \int_{ 0 }^{ T } \left\|Q \I_{\{Z_t>0\}}\I_{\{Z_t=0\}}e_k\right\|^2 dt=0
   \end{equation}
   as $n\to\infty$. We set $\tilde{\I}^n_t:=\I_{\{Z^n_t>0\}}\I_{\{Z_t=0\}}$, $t \in [0,T]$. Then~\eqref{equ_convergence_of_qie} and the equality
   \begin{align*}
     \int_{ 0 }^{ T } \left\|Q \tilde{\I}^n_te_k\right\|^2dt&= \sum_{ l=1 }^{ \infty } \int_{ 0 }^{ T } \mu_l^2 \left\langle \tilde{\I}^n_te_k , e_l \right\rangle^2dt  
   \end{align*}
   imply 
   \[
     \int_{ 0 }^{ T } \left\langle \tilde{\I}^n_te_k , e_k  \right\rangle^2 dt \to 0, \quad n\to\infty, 
   \]
   for every $k\geq 1$ such that $\mu_k>0$. So, by the H\"older inequality,
   \[
     \left(\int_{ 0 }^{ T } \int_{ 0 }^{ 1 } \tilde{\I}^n_t(u)e_k^2(u) dtdu  \right)^2\leq T \int_{ 0 }^{ T } \left\langle \tilde{\I}^n_t e_k, e_k \right\rangle^2dt \to 0, \quad n\to\infty. 
   \]
   Taking into account the equality $\tilde{\I}^n_t=\left( \tilde{\I}^n_t \right)^2$, $t \in[0,T]$, we can conclude that 
   \begin{equation} 
  \label{equ_convergence_of_til_i_ek}
  \tilde{\I}^n e_k \to 0 \quad \mbox{in}\ \ \LL{ \L_2}, \quad n\to\infty,
   \end{equation}
   for every $k\geq 1$ such that $\mu_k>0$.

   We claim that $\chi \tilde{\I}^n$, $n\geq 1$, converges to $0$ in $\LL{ \L_2}$ as $n\to\infty$. Indeed, by convergence~\eqref{equ_convergence_of_til_i_ek} and the dominated convergence theorem, 
   \begin{equation} 
  \label{equ_convergence_of_chii}
     \left\|\chi \tilde{\I}^n\right\|_{\L_2,T}^2=\sum_{ k=1 }^{ \infty } \mu_k^2\int_{ 0 }^{ T } \int_{ 0 }^{ 1 }   \tilde{\I}^n_t(u) e_k^2(u) dtdu \to 0, \quad n\to\infty.    
   \end{equation}

   Next, since $\I_{\{Z^n>0\}} \to \sigma$ in the weak topology of $\LL{ \L_2}$ as $n\to\infty$, and $\I_{\{Z>0\}}$, $\I_{\{Z=0\}}$ are uniformly bounded, we trivially obtain that
   \begin{equation} 
  \label{equ_convergence_sigma_i}
  \I_{\{Z^n>0\}}\I_{\{Z>0\}} \to \sigma \I_{\{Z>0\}}, \quad \tilde{\I}^n=\I_{\{Z^n>0\}}\I_{\{Z=0\}} \to \sigma \I_{\{Z=0\}},
   \end{equation}
   in the weak topology of $\LL{ \L_2}$ as $n\to\infty$. Using the fact that 
   \[
     \I_{(0,+\infty)}(x_n) \I_{(0,+\infty)}(x) \to \I_{(0,+\infty)}(x) \quad \mbox{as}\ \ x_n \to x \ \ \mbox{in}\ \ \R,
   \]
   and the uniqueness of a weak limit, we get 
   \begin{equation} 
  \label{equ_sigma_i_pos}
  \sigma \I_{\{Z>0\}}=\I_{\{Z>0\}}.
   \end{equation}
   Since $\chi \in \L_2$, convergence~\eqref{equ_convergence_sigma_i} yields 
   \[
     \int_{ 0 }^{ T } \int_{ 0 }^{ 1 } \chi(u) \tilde{\I}^n_t(u)dtdu  \to \int_{ 0 }^{ T } \int_{ 0 }^{ 1 } \chi(u) \sigma_t(u) \I_{\{Z_t=0\}}(u)dtdu,\quad n\to\infty.
   \]
   On the other hand side, $\chi \tilde{\I}^n \to 0$ in $\LL{ \L_2}$, by~\eqref{equ_convergence_of_chii}. Hence 
   \[
     \chi \sigma \I_{\{Z=0\}}=0.
   \]
   The latter equality and~\eqref{equ_sigma_i_pos} yield
   \[
     \chi \sigma=\chi \sigma \I_{\{Z>0\}}+\chi \sigma \I_{\{Z=0\}}=\chi\I_{\{Z>0\}} \quad \mbox{in}\ \ \LL{ \L_2}
   \]
   that is equivalent to equality~\eqref{equ_equality_for_zsig}. This completes the proof of the lemma.
\end{proof}

\subsubsection*{Acknowledgement} 

The author is very grateful to Prof. Dr. Dorogovtsev for valuable discussions during the work on this paper.

\appendix

\section{Auxiliary statements}%
\label{sec:auxiliary_statememts}

In this section, we will prove the following lemma.

\begin{lemma} 
  \label{lem_quadratic_variation_of_rv_semimartingale}
  Let $\xi_k(t)$, $t \geq 0$, $k \in [2]$, be continuous real valued semimartingales with respect to the same filtration. Let also the quadratic variations equal
  \[
    [\xi_k,\xi_l]_t=\int_{ 0 }^{ t } \sigma_{k,l}(s)ds, \quad t\geq 0, \ \ k,l \in [2]. 
  \]
  Then almost surely for all $k,l \in [2]$
  \[
    [\xi_k,\xi_l]_t=\int_{ 0 }^{ t } \sigma_{k,l}(s)\I_{\left\{ \xi_k(s)\not= 0 \right\}}\I_{\left\{ \xi_l(s)\not= 0 \right\}}ds,\quad t \geq 0.
  \]
\end{lemma}

\begin{proof} %
  By Theorem~22.5~\cite{Kallenberg:2002}, one has for $k \in [2]$ almost surely
  \begin{align*}
    \int_{ 0 }^{ t } \sigma_{k,k}(s)\I_{\left\{ \xi_k(s)=0 \right\}}ds&=  \int_{ 0 }^{ t } \I_{\left\{ 0 \right\}}(\xi_k(s))d[\xi_k]_s\\
    &= \int_{ -\infty }^{ +\infty } \I_{\left\{ 0 \right\}}(x)L_t^{k,x}dx=0,\quad t\geq 0,
  \end{align*}
  where $L_t^{k,x}$, $t \geq 0$, $x \in \R$, is the local time of $\xi_k$. Applying the Cauchy-type inequality~\cite[Proposition~17.9]{Kallenberg:2002}, we estimate almost surely for every $t\geq 0$
  \begin{align*}
    \int_{ 0 }^{ t } |\sigma_{1,2}(s)|&\I_{\left\{ \xi_1(s)=0 \right\}}ds\leq \int_{ 0 }^{ t } \sigma_{1,1}(s)\I_{\left\{ \xi_1(s)=0 \right\}}ds \int_{ 0 }^{ t } \sigma_{2,2}(s)ds =0.
  \end{align*}
  Similarly, we get 
  \[
    \int_{ 0 }^{ t } |\sigma_{1,2}(s)|\I_{\left\{ \xi_2(s)=0 \right\}}ds=0, \quad t\geq 0,\ \ \mbox{a.s.} 
  \]
  These equalities trivially yield the statement of the lemma.
\end{proof}

\begin{lemma} 
  \label{lem_convergence_of_lambdan}
  Let $\lambda$ be a non-negative function from $\L_2$, $Q$ be a non-negative definite self-adjoint Hilbert-Schmidt operator on $\L_2$, $\chi^2$ be defined by~\eqref{equ_function_chi} and 
  \[
    \lambda^n=\sum_{ k=1 }^{ n } n \langle \lambda , \pi_k^n \rangle \I_{\left\{ q_{k,k}^n>0 \right\}}\pi_k^n,\quad n\geq 1,
  \]
  where $q_{k,k}^n=n \|Q \pi_k^n\|^2$. If $\lambda \I_{\left\{ \chi>0 \right\}}=\lambda$ a.e., then $\lambda^n \to \lambda$ in $\L_2$ as $n\to\infty$.
\end{lemma}

\begin{proof} %
  Denote 
  \[
    \tilde{\lambda}^n:=\pr^n \lambda=\sum_{ k=1 }^{ n } n \langle \lambda , \pi_k^n \rangle \pi_k^n,\quad n\geq 1.
  \]
  In this proof, functions from $\L_2$ will be considered as random elements on the probability space $([0,1],\B([0,1]),\leb_1)$, where $\B([0,1])$ is the Borel $\sigma$-algebra on $[0,1]$. We remark that $\tilde{\lambda}^n$ is the conditional expectation $\E{\lambda|\cS^n}$ determined on that probability space, where $\cS^n=\sigma\left\{ \pi_k^n,\ k \in [n] \right\}$.  By Proposition~1~\cite{Alonso:1998}, $\tilde{\lambda}^n \to \lambda$ in $\L_2$ as $n \to \infty$. In particular, $\tilde{\lambda}^n$ converges to $\lambda$ in probability as $n \to \infty$.

  Let
  \begin{align*}
    q^n:&= \sum_{ k=1 }^{ n } nq^n_{k,k}\pi_k^n=\sum_{ k=1 }^{ n } n^2 \|Q \pi_k^n\|^2 \pi_k^n=\sum_{ k=1 }^{ n } \left( n^2 \sum_{ l=1 }^{ \infty }\mu_l^2 \langle e_l , \pi_k^n \rangle^2 \right)\pi_k^n\\
    &= \sum_{ l=1 }^{ \infty } \mu_l^2\left( \sum_{ k=1 }^{ n } n^2 \langle e_l , \pi_k^n \rangle^2 \pi_k^n \right)=\sum_{ l=1 }^{ \infty } \mu_l^2\left( \pr^ne_l \right)^2, \quad n\geq 1.
  \end{align*}
  Remark that $\pr^n e_l \to e_l$ in probability as $n\to\infty$ for all $l\geq 1$. 

  We fix a subsequence $N \subset \N$. Then, by Lemma~4.2~\cite{Kallenberg:2002}, there exists a subsequence $N' \subset N$ such that $\tilde{\lambda}^n \to \lambda$ a.s. along $N'$. Using Lemma~4.2~\cite{Kallenberg:2002} again and the diagonalisation argument, we can find a subsequence $N'' \subset N'$ such that $\pr^n e_l \to e_l$ a.s. along $N''$ for all $l\geq 1$. By Fatou's lemma, 
  \[
    \varliminf_{ N'' \ni n\to\infty }q^n\geq \sum_{ l=1 }^{ \infty } \mu_l^2 e_l^2=\chi^2\quad \mbox{a.s.}
  \]
  This inequality and the lower semicontinuity of the map $\R\ni x \mapsto \I_{(0,+\infty)}(x)$ yield 
  \[
    \varliminf_{ N''\ni n\to\infty }\I_{\left\{ q^n>0 \right\}}\geq \I_{\left\{ \chi^2>0 \right\}}=\I_{\left\{ \chi>0 \right\}} \quad \mbox{a.s.}
  \]
  Consequently, using the equality 
  \begin{equation} 
  \label{equ_lambda_and_lambda_tilda}
    \lambda^n=\sum_{ k=1 }^{ n } n \langle \lambda , \pi_k^n \rangle \I_{\left\{ q_{k,k}^n>0 \right\}}\pi_k^n=\tilde{\lambda}^n \I_{\left\{ q^n>0 \right\}},
  \end{equation}
  and the convergence $\tilde{\lambda}^n \to \lambda$ a.s. along $N''$,
  we obtaing
  \[
    \varliminf_{ N''\ni n\to\infty }\lambda^n=\varliminf_{ N''\ni n\to\infty }\tilde{\lambda}^n \I_{\left\{ q^n>0 \right\}}\geq \lambda \I_{\left\{ \chi>0 \right\}}=\lambda \quad \mbox{a.s.}
  \]
  By~\eqref{equ_lambda_and_lambda_tilda}, we also have 
  \[
    \varlimsup_{ N''\ni n\to\infty }\lambda^n\leq \varlimsup_{ N''\ni n\to\infty }\tilde{\lambda}^n=\lambda \quad \mbox{a.s.}
  \]
  This implies the convergence $\lambda^n \to \lambda$ a.s. along $N''$, and hence, $\lambda^n \to \lambda$ in probability as $n\to\infty$, by Lemma~4.2~\cite{Kallenberg:2002}. We also remark that $\lambda^n\leq \tilde{\lambda}^n$, $n\geq 1$, and $\tilde{\lambda}^n \to \lambda$ in $\L_2$. Hence, dominated convergence Theorem~1.21~\cite{Kallenberg:2002} implies that $\|\lambda^n\| \to \|\lambda\|$. By Proposition~4.12~\cite{Kallenberg:2002}, $\lambda^n \to \lambda$ in $\L_2$ as $n\to\infty$.
\end{proof}

\begin{lemma} 
  \label{lem_norm_of_composition_of_hilbert_shmidt_operators}
  Let $A \in \HS$, and $B_i$, $i=1,2$, be bounded operators on $\L_2$. Then $AB_i \in \HS$, $i=1,2$, and
  \[
    \langle AB_1 , AB_2 \rangle_{\HS}=\sum_{ n=1 }^{ \infty } \nu_n^2 \langle B^*_1\eps_n,B^*_2\eps_n \rangle,
  \]
  where $\{ \eps_n,\ n\geq 1 \}$ and $\left\{ \nu_n^2,\ n\geq 1 \right\}$ are eigenvectors and eigenvalues of $A^*A$, respectively.
\end{lemma}

\begin{proof} %
  Set $A^n:=\sum_{ l=1 }^{ n } \nu_l \eps_l \odot \eps_l $, $n\geq 1$. Then it is easily seen that the sequence $\left\{A^n\right\}_{n\geq 1}$ converges to $\sqrt{ A^*A }=\sum_{ l=1 }^{ \infty } \nu_l \eps_l \odot \eps_l$ in $\HS$. Hence
  \begin{align*}
    \langle AB_1 , AB_2 \rangle_{\HS}&= \sum_{ k=1 }^{ \infty } \langle AB_1 \eps_k , AB_2 \eps_k \rangle=\sum_{ k=1 }^{ \infty } \langle A^*AB_1 \eps_k , B_2 \eps_k \rangle\\
    &= \left\langle \sqrt{ A^*A }B_1 , \sqrt{ A^*A }B_2 \right\rangle_{\HS}= \lim_{ n\to\infty }\left\langle A_nB_1,A_nB_2 \right\rangle_{\HS}\\
    &= \lim_{ n\to\infty }\sum_{ k=1 }^{ \infty } \left\langle A_nB_1\eps_k , A_nB_2 \eps_k \right\rangle=  \lim_{ n\to\infty }\sum_{ k=1 }^{ \infty } \sum_{ l=1 }^{ n } \nu_l^2 \langle B_1 \eps_k , \eps_l \rangle \langle B_2 \eps_k , \eps_l \rangle\\
    &= \sum_{ l=1 }^{ \infty } \sum_{ k=1 }^{ \infty }\nu_l^2 \langle \eps_k , B^*_1 \eps_l \rangle \langle \eps_k , B^*_2 \eps_l \rangle = \sum_{ l=1 }^{ \infty } \nu_l^2 \langle B^*_1 \eps_l , B^*_2 \eps_l \rangle.
  \end{align*}
\end{proof}

\begin{lemma} 
  \label{lem_convergence_of_composition_in_hilbert_shmidt_space}
  Let $A \in \HS$, and a sequence of bounded operators $B_n$, $n\geq 1$, in $\L_2$ converge pointwise to an operator $B$, that is, for every $\varphi \in \L_2$ $B_n \varphi \to B \varphi$ in $\L_2$ as $n\to\infty$. Then $B$ is bounded and $AB_n^* \to AB^*$ in $\HS$ as $n\to\infty$.
\end{lemma}

\begin{proof} %
  We first note that norms $\|B_n\|$, $n\geq 1$, are uniformly bounded, by the Banach-Stheinhaus theorem. Consequently, $B$ is a bounded operator on $\L_2$.
  
  Next, we will show that $\left\{AB_n^*\right\}_{n\geq 1}$ converges to $AB^*$ in the weak topology of $\HS$. Let $\{ \eps_n,\ n\geq 1 \}$ and $\{ \nu_n^2,\ n\geq 1 \}$ are eigenvectors and eigenvalues of $A^*A$. Then for every $k,l\geq 1$
  \begin{align*}
    \left\langle AB_n^*,\eps_k \odot \eps_l \right\rangle_{\HS}&= \langle AB_n^*\eps_l , \eps_k \rangle = \langle \eps_l , B_nA^*\eps_k \rangle\\
    &\to \langle \eps_l , BA^*\eps_k \rangle=\left\langle AB^*,\eps_k \odot \eps_l \right\rangle_{\HS} \quad \mbox{as}\ \ n\to\infty.
  \end{align*}
Since $\spann\{ \eps_k\odot \eps_l,\ k,l \geq 1 \}$ is dense in $\HS$ and $\left\|AB_n^*\right\|_{\HS}\leq \|A\|_{\HS}\|B_n^*\|$, $n \geq 1$, is uniformly bounded, the sequence $\left\{AB_n^*\right\}_{n\geq 1}$ converges to $AB^*$ in the weak topology of $\HS$. By the dominated convergence theorem, the uniform boundedness of the norms of $\|B_n\|$, $n\geq 1$, and Lemma~\ref{lem_norm_of_composition_of_hilbert_shmidt_operators}, we obtain 
  \[
    \left\|AB_n^*\right\|_{\HS}^2=\sum_{ l=1 }^{ \infty } \nu_l^2 \|B_n \eps_l\|^2 \to \sum_{ l=1 }^{ \infty } \nu_l^2 \|B \eps_l\|^2 =\left\|AB^*\right\|_{\HS}^2 \quad \mbox{as}\ \ n\to\infty.
  \]
 This implies that $\left\{AB_n^*\right\}_{n\geq 1}$ converges to $AB^*$ in the strong topology of $\HS$.

\end{proof}

Let $\HS^{p,sa}$ be a closed subset of $\HS$ consisting of non-negative definite self-adjoint operators. Consider 
\[
  \rB\left(\HS^{p,sa}\right):=\left\{ L \in \rB\left(\HS\right):\ L \in \HS^{p,sa}\ \ \mbox{a.e.}  \right\}.
\]
Remark that a non-negative definite self-adjoint operator $A$ on $\L_2$ has the square root, i.e. there exists a unique non-negative definite self-adjoint operator $\sqrt{ A }$ on $\L_2$ such that $\left(\sqrt{ A }\right)^2=A$. This trivially follows from the spectral theorem.

\begin{lemma} 
  \label{lem_square_root_operator}
    \begin{enumerate}
      \item [(i)] The set $\rB\left( \HS^{p,sa} \right)$ is closed in $\rB\left(\HS\right)$.
  
      \item [(ii)] For every $r>0$ the set 
        \[
  	S_r:=\left\{ L \in \rB\left( \HS^{p,sa} \right):\ \int_{ 0 }^{ T } \left\|\sqrt{ L_t }\right\|^2_{\HS}dt\leq r  \right\}
        \]
        is closed in $\rB\left(\HS\right)$.
  
      \item [(iii)] For every $r>0$ the map $\Phi^r:S_r \to \rB\left( \HS^{p,sa} \right)$ defined as 
        \[
  	\Phi^r_t(L)=\sqrt{ L_t }, \quad t \in[0,T], \quad L \in S_r,
        \]
        is Borel measurable.
    \end{enumerate}
\end{lemma}

\begin{proof} %
  Let $L^n$, $n\geq 1$, be a sequence from $\rB\left( \HS^{p,sa} \right)$ which converges to $L$ in $\rB\left(\HS\right)$. We take arbitrary $t \in[0,T]$ and $\varphi,\psi \in \L_2$ and consider 
  \begin{align*}
    \int_{ 0 }^{ t } \left\langle L_s \varphi , \psi \right\rangle ds&= \int_{ 0 }^{ t } \left\langle L_s, \psi\odot \varphi \right\rangle_{\HS}ds=\lim_{ n\to\infty }\int_{ 0 }^{ t } \left\langle L_s^n, \psi \odot \varphi \right\rangle_{\HS}ds \\
    &=\lim_{ n\to\infty }\int_{ 0 }^{ t } \left\langle L_s^n \varphi, \psi \right\rangle ds = \lim_{ n\to\infty }\int_{ 0 }^{ t } \left\langle \varphi , L_s^n \psi \right\rangle ds = \int_{ 0 }^{ t } \left\langle \varphi , L_s \psi \right\rangle ds.
  \end{align*}
  Due to the density of the set $\spann\left\{ \I_{[0,t]} \varphi\odot \psi,\ t \in[0,T],\ \varphi,\psi \in \L_2 \right\}$ in $\LL{ \HS}$, we obtain that $L$ is self-adjoint a.e. Similarly, one can show that $L$ is non-negative definite. Hence, $\rB\left(\HS^{p,sa}\right)$ is closed.

  Next we prove (ii). Take a sequence $L^n$, $n\geq 1$, from $S_r$ which converges to $L$ in $\rB\left(\HS\right)$. We remark that $L \in \rB\left(\HS^{p,sa}\right)$ due to (i). Then 
  \begin{align*}
    \int_{ 0 }^{ T } \left\|\sqrt{ L_t }\right\|_{\HS}^2dt&= \int_{ 0 }^{ T }\left[ \sum_{ k=1 }^{ \infty } \left\|\sqrt{ L_t }e_k\right\|^2\right]dt =\sum_{ k=1 }^{ \infty } \int_{ 0 }^{ T } \left\langle L_te_k , e_k \right\rangle dt  \\
    &\leq \varliminf_{ n\to\infty }\sum_{ k=1 }^{ \infty } \int_{ 0 }^{ T } \left\langle L^n_te_k , e_k \right\rangle dt=\varliminf_{ n\to\infty }\int_{ 0 }^{ T } \left\|\sqrt{ L_t^n }\right\|_{\HS}^2dt\leq r,
  \end{align*}
  by Fatou's lemma and the fact that $\int_{ 0 }^{ T } \left\langle L_t^ne_k , e_k \right\rangle \to \int_{ 0 }^{ T } \left\langle L_te_k , e_k \right\rangle  $, $n\to\infty$, for all $k\geq 1$. Thus, $S_r$ is closed.
  
  In order to check (iii), we first remark that it is enough to show that for every $t \in[0,T]$ and $\varphi,\psi \in \L_2$, the map 
  \begin{equation} 
  \label{equ_square_root_map_with_phi_psi}
  S_r\ni L\mapsto \int_{ 0 }^{ T } \left\langle \Phi^r_s(L),\I_{ [0,t] }(s)\psi\odot \varphi \right\rangle_{\HS}ds= \int_{ 0 }^{ t } \left\langle \Phi^r_s(L)\varphi,\psi \right\rangle ds\in \R 
  \end{equation}
  is Borel measurable. By Theorem~1.2~\cite{Vakhania:1987}, the Borel $\sigma$-algebra on $\rB\left(\HS\right)$ coincides with $\sigma$-algebra of all Borel measurable sets of $\LL{ \HS}$ contained in the ball $\rB\left(\HS\right)$. Consequently, it is enough to show that map~\eqref{equ_square_root_map_with_phi_psi} is Borel measurable as a map from $S_r$ to $\R$, where $S_r$ is embedded with the strong topology of $\LL{ \HS}$. But then map~\eqref{equ_square_root_map_with_phi_psi} 
  \[
    S_r\ni L\mapsto\int_{ 0 }^{ t } \left\langle \Phi^r_s(L)\varphi , \psi \right\rangle=\int_{ 0 }^{ t } \left\langle L_s \varphi , L_s \psi \right\rangle ds
  \]
  is continuous, and, thus, Borel measurable. This finishes the proof of the lemma.
\end{proof}

Let the basis $\left\{\tilde{e}_k,\ k\geq 1\right\}$ in $\L_2$ be defines as in Section~\ref{sub:a_property_of_quadratic_variation_of_some_semi_martingales}, that is, $\tilde{e}_1(u)=1$, $u \in[0,1]$, and $\tilde{e}_k(u)=\sqrt{ 2 }\cos \pi(k-1)u$, $u \in[0,1]$, $k\geq 2$. For $h \in \L_2$ we define 
\[
  \dot{h}=\sum_{ n=1 }^{ \infty } \langle h , \tilde{e}_n \rangle \tilde{e}'_n,
\]
if the series converges in $\L_2$. We remark that $\langle \dot{h} , \varphi \rangle=-\langle h , \varphi' \rangle$ for every $\varphi \in \Cf^1[0,1]$ with $\varphi(0)=\varphi(1)=0$.

\begin{lemma} 
  \label{lem_truncation}
  Let $h \in \L_2$ be non-negative and $\dot{h}$ exist. Then $\dot{h}\I_{\left\{ 0 \right\}}(h)=0$ a.e.
\end{lemma}

\begin{proof} %
  We consider for every $\eps>0$ the function $\psi_{\eps}(x)=\sqrt{ x^2+\eps^2 }-\eps$, $x \in \R $. Then $\psi_{\eps}$ is continuously differentiable, $\psi_{\eps}(0)=0$ and $\psi_{\eps}(x) \to |x|$ as $\eps \to 0+$ for all $x \in \R $. Moreover, $|\psi'_{\eps}(x)|\leq 1$ and $\psi'_{\eps}(x)\to \sgn(x)$ for all $x \in \R $. Take any function $\varphi \in \Cf[0,1]$ satisfying $\varphi(0)=\varphi(1)=0$. By the dominated convergence theorem, it is easily seen that 
  \[
    \langle \psi_{\eps}(h) , \varphi' \rangle=-\langle \psi'_{\eps}(h)\dot{h} , \varphi \rangle.
  \]
  Making $\eps \to 0+$, and using the non-negativity of $h$, we have
  \[
    -\langle \dot{h} , \varphi \rangle=\langle h , \varphi' \rangle=-\langle \I_{ (0,+\infty) }(h)\dot{h} , \varphi \rangle.
  \]
  Since $\varphi$ was arbitrary, we can conclude that $\dot{h}=\dot{h}\I_{(0,+\infty)}(h)$ a.e. This completes the proof of the lemma. 

\end{proof}

\begin{remark} %
  The same statement of Lemma~\ref{lem_truncation} remains true if the ``cos'' basis is replaced by the ``sin'' basis $\tilde{e}_k=\sqrt{ 2 }\sin \pi ku$, $u \in [0,1]$, $k\geq 1$.
\end{remark}

\providecommand{\bysame}{\leavevmode\hbox to3em{\hrulefill}\thinspace}
\providecommand{\MR}{\relax\ifhmode\unskip\space\fi MR }
\providecommand{\MRhref}[2]{%
  \href{http://www.ams.org/mathscinet-getitem?mr=#1}{#2}
}
\providecommand{\href}[2]{#2}


\end{document}